\documentclass[10 pt,a4paper]{article}
\usepackage[toc,page]{appendix}
\usepackage{amsmath}
\usepackage{tocbibind}
\usepackage{stmaryrd}   
\usepackage{amsmath}
\usepackage{verbatim}
\usepackage{mathrsfs}
\usepackage{mathtools}
\usepackage{amsthm} 
\usepackage{amssymb} 
\usepackage{mdframed}
\usepackage[margin=2.5 cm]{geometry}
\usepackage{enumerate}
\usepackage{bm}
\usepackage[colorlinks={true},linkcolor={blue},citecolor={blue}, urlcolor={blue}]{hyperref}
\usepackage[noabbrev,nameinlink]{cleveref}
\usepackage{comment}
\usepackage{tikz}
\usepackage{tikz-cd}
\usetikzlibrary{intersections}
\usetikzlibrary{spy} 
\theoremstyle{plain}

\newtheorem{thm}{Theorem}[section]
\newtheorem{prop}[thm]{Proposition}

\newtheorem{lem}[thm]{Lemma}
\theoremstyle{definition}
\newtheorem{defn}[thm]{Definition}

\newtheorem{expl}[thm]{Example}

\newtheorem{rem}[thm]{Remark}

\newtheoremstyle{named}%
    {}{}{\itshape}{}{\bfseries}{.}{.5em}{\thmnote{#3}}
\theoremstyle{named}
\newtheorem*{namedtheorem}{Theorem}

\newcommand\bref[3][blue]{%
    \begingroup%
    \hypersetup{linkcolor=#1}%
    \hyperlink{#2}{#3}%
    \endgroup}

\newcommand{\sC}{\mathcal{C}}

\newcommand{\sG}{\mathcal{G}}

\newcommand{\mK}{\mathbb{K}}

\newcommand{\ord}{\mathrm{ord}\,}

\DeclareMathOperator{\Ind}{Ind}
\DeclareMathOperator{\spec}{Spec}
\DeclareMathOperator{\Gal}{Gal}
\DeclareMathOperator{\id}{id}

\begin{document}
 \title{On the generalisation of Roth's theorem}
\author{Paolo Dolce\and
  Francesco Zucconi}
\date{}
\maketitle
\makeatletter
\@starttoc{toc}
\makeatother

\abstract{We present two generalisations of Roth's approximation theorem on proper adelic curves, assuming some technical conditions on the  behavior of the logarithmic absolute values. We illustrate how tightening such assumptions makes our inequalities stronger. As special cases  we recover Corvaja's results  \cite{C} for fields admitting a product formula, and Vojta's ones  \cite{vo} for arithmetic function fields.}

\setcounter{section}{-1}

\section{Introduction}
\subsection{History}\label{hi}
The celebrated Roth's theorem proved in \cite{roth} asserts that the irrationality measure of a real algebraic number is $2$. An equivalent, but more detailed statement is the following:
\begin{thm}[Roth]
Let $\alpha$ be a real algebraic number and let $\varepsilon>0$ be a real number. Then there exists a real constant $C(\alpha,\varepsilon)>0$ such that for every pair of coprime integers $(p,q)$ with $q>C(\alpha,\varepsilon)$, it holds that:
$$
\left|\alpha-\frac{p}{q}\right|> q^{-2-\varepsilon}
$$
\end{thm}

The above statement can be  naturally generalised in different directions by considering a number field $K$ instead of $\mathbb Q$ and the simultaneous approximation of the elements $\alpha_1,\ldots,\alpha_n$ algebraic over  $K$ by an element of $K$ with respect to different valuations (see \cite[Chapter 7]{lang}). Actually, the statement proved in \cite[Chapter 7]{lang} holds for fields that are more general than  number fields.  At the moment we don't have an \emph{effective} version of Roth's theorem  (i.e. a bound for the constant $C(\alpha,\varepsilon)$, see \cite{C1}), but we have a \emph{quantitative}  version (i.e.  bounds, in terms of $\alpha$ and $\varepsilon$, on the number of good approximants  see  \cite{bom}, \cite{bom1}, \cite{ever}).

Let $k$ be a field of characteristic $0$ and let $\mathcal V_k$ be a  set in bijection with a family of absolute values of $k$. The bijection is denoted by $v\mapsto |\cdot|_v$, for any $v\in\mathcal V_k$ and we don't put any restriction on the cardinality of $\mathcal V_k$, therefore we recall that a sum  over an uncountable  set  is defined to be the supremum of the sums over all finite subsets. The couple  $(k,\mathcal V_k)$ \emph{satisfies the product formula} if  for any element $a\in k^\times$ the series $\sum_{v\in\mathcal V_k} \log|a|_v$ converges absolutely and moreover $\sum_{v\in\mathcal V_k} \log|a|_v=0$. In this setting one also has a natural notion of \emph{logarithmic height} for $a\in k^\times$:
$$h(a):=\sum_{v\in\mathcal V_K}\log^+|a|_v\,,$$
and we set $H(a):=e^{h(a)}$.  Roth's theorem was generalised by Corvaja in  \cite{C}   for any couple $(k,\mathcal V_k)$ satisfying the product formula:
\begin{thm}\cite[Corollaire 1]{C}.\label{CorCorv}
Assume that $(k,\mathcal V_k)$ satisfies the product formula. Let $\alpha_1,\ldots,\alpha_n$ be distinct elements algebraic on $k$, and let $|\cdot|_{v_1},\ldots,|\cdot|_{v_n}$ be distinct absolute values of $k$, with $v_1,\ldots, v_n\in \mathcal V_k$. For any $i=1,\ldots,n$ let's fix an appropriately  normalised extension of  $|\cdot|_{v_i}$ to $k(\alpha_i)$ (and by abuse of notation denote it with the same symbol $|\cdot|_{v_i}$). Then for any  $\varepsilon>0$ there exists a constant $C=C(\mathcal V_k,\alpha_1,\ldots,\alpha_n, v_1,\ldots, v_n,\varepsilon)>0$ such that for all $\beta\in k$ with $H(\beta)>C$ it holds that:
$$\sum^n_{i=1}\log|\alpha_i-\beta|_{v_i}>-(2+\varepsilon)h(\beta)\,.$$
\end{thm}
Notice that there are no further assumptions on the field $k$, which might be for instance a function field; therefore, theorem \ref{CorCorv} is a unifying result, as well as a generalisation of the classical Roth's theorem. Moreover, in the same paper Corvaja obtained also a quantitative version of Theorem \ref{CorCorv} (see \cite[Corollaire 3.7]{C}). 
 
An \emph{arithmetic function field},  is a  finitely generated field over $\mathbb Q$. The ``geometrisation'' of these fields in terms of Arakelov geometry is wonderfully explained in \cite{Mor}. Recently,  Vojta in \cite{vo} proved a version of Roth's theorem for arithmetic function fields with a big polarisation. For obvious reasons this can be considered as a ``higher dimensional'' generalisation of Roth's theorem.  We  conclude this short historic overview by explaining the statement of Vojta's  result. A big polarisation of an arithmetic function field  $K$  consists of a couple $(X,\overline{\mathcal L})$ where:
\begin{itemize}
\item[$(i)$]$X\to \spec\mathbb Z$ is a normal arithmetic variety whose function field is isomorphic to $K$.
\item[$(ii)$] Denote with $d$ the relative dimension of $X$ over  $\spec\mathbb Z$; then $\overline{\mathcal L}=\{\overline{\mathscr L}_1,\ldots\overline{\mathscr L}_d\}$  is a collection of hermitian, arithmetically nef and big line bundles.
\end{itemize}
Now  fix an arithmetic function field  $K$ with a big polarisation; then we can define a geometric height function $h_{\overline{\mathcal L}}$ for a prime divisor $Y$ as the Arakelov intersection number of the hermitian  line bundles $\overline{\mathscr L}_1,\ldots,\overline{\mathscr L}_d$ restricted to $Y$. We can now define a non-archimedean absolute value on $K$, associated to $Y$:
$$\vert a\vert_Y:=e^{-h_{\overline{\mathcal L}}(Y)\ord_Y(a)}\quad\forall a\in K\,.$$
Moreover, for any closed point $p
\in X_{\mathbb C}=X\times_{\spec\mathbb Z}\spec \mathbb C$ that doesn't come from the base change of a divisor on $X$ we have the following archimedean absolute value:
$$|a|_p:=\sqrt{a(p)\overline {a(p)}}\quad\forall a\in K\,.$$
By putting all together, we get a set of absolute values $M_K$ which turns out to be a measure space with a measure that we denote with $\mu$. The notion of  product formula holds true for the couple $(K, M_K)$ in the following form:
$$\int_{M_K}\log \vert a\vert_\nu d\mu(\nu)=0\,,\quad\forall a\in K^\times$$
Moreover there is a notion of height for any element of $a\in K^\times$:
$$
h_{K}(a):=\int_{M_K} \log^+|a|_{v}\, d\mu(v)\,.
$$
We set $H_K(a):=e^{h_{K}(a)}$. One of the $4$ equivalent versions of Vojta's generalisation of Roth's theorem is given below.
\begin{thm}\cite[Theorem 4.5]{vo}\label{vojta}
Consider a couple $(K,M_K)$ where $K$ is an arithmetic function field with a big polarisation such that $\overline{\mathscr L}_1=\ldots=\overline{\mathscr L}_d$, and $M_K$ is a set of absolute values as explained before. Let $S$ be a subset of $M_K$ of finite measure and fix some distinct elements $\alpha_1,\dots ,\alpha_n\in  K$. Then for any $\varepsilon>0$ and any $c\in\mathbb R$ there exists a real constant $C>0$ (depending on the fixed data) such that for any $\beta\in  K$  with  $h_{K}(\beta)>C$   
it holds that:
\begin{equation}\label{our_roth_encore}
\int_{S}\,\min_{1\le i\le n}\left(\log^-\vert \beta-\alpha_i\vert_\omega\right) d\mu(\omega)>-(2+\varepsilon)h_K(\beta)+c
\end{equation}
\end{thm}
\subsection{Results in this paper}
The goal of this paper is to generalise Roth's theorem  in a wider framework which includes Corvaja's and Vojta's settings.

The theory of  adelic curves introduced by Chen and Moriwaki in \cite{Ch-Mo} provides such natural framework:  an adelic curve $\mathbb X$ consists of a field $\mathbb K$ of characteristic $0$ and a measure space $(\Omega,\mathcal A, \mu)$ endowed with a function  $\phi:\omega\mapsto \vert\cdot \vert_{\omega}$ that maps $\Omega$ into the set of places of $\mathbb K$ and such that $\omega\mapsto\log|a|_\omega$ is  in $L^1(\Omega,\mu)$ for any $a\in\mathbb K^\times$. On $\mathbb X$ we have a well defined notion of height $h_{\mathbb X}$, and moreover a ``product formula'' which is expressed as an integral over $\Omega$. The adelic curves satisfying this product formula are called \emph{proper}. The fields with a product formula of $\cite{C}$ are trivially proper adelic curves when the set of places is endowed with the counting measure. Moreover in \cite[Section 3]{vo} it is shown that arithmetic function fields can be endowed with  a structure of proper adelic curve.

In this paper we prove two generalisations of Roth's theorem for proper adelic curves. The main results will be stated in terms of some inequalities involving the integral of the measurable functions $\omega\mapsto \log^-|\beta-\alpha_i|_\omega$  where $\alpha_1,\ldots, \alpha_n\in\mathbb K$ are the elements we want to approximate by an approximant $\beta\in\mathbb K$. Unfortunately, the bare definition of proper adelic curves is too general to give any meaningful result in the direction of Roth's approximation theorem, since  the functions $\log^-|\beta-\alpha_i|$ can be in principle very ``wild''. Therefore, we have to impose some analytic conditions on such functions in order to get the desired Roth's theorems. We will see that these assumptions are not too artificial, in fact we recover Theorems \ref{CorCorv} and \ref{vojta} as special cases.

The first condition we impose  on our adelic curves is the \emph{strong $\mu$-equicontinuity} (see definition \ref{st_muequi}): roughly speaking it  means that for any set of finite measure $S\subset\Omega$ one requires for the functions $\omega\mapsto \log|\beta|_{\omega}$, to be ``equicontinuous'' almost everywhere  on $S$. Under such hypothesis we prove the following theorem:

\begin{namedtheorem}[\hypertarget{thm:(A)}{Theorem (A)}]
Let $\mathbb X=(\mathbb K,\Omega, \phi)$ be a proper adelic curve satisfying the strong $\mu$-equicontinuity condition. Let $S=S_1\sqcup S_2\sqcup\ldots\sqcup S_n\subseteq\Omega $ be the disjoint union of subsets of finite measure. Fix some distinct elements $\alpha_1,\dots ,\alpha_n\in\mathbb K$, then for any $\varepsilon>0$ there exists a real constant $C>0$ (depending on the fixed data) such that for any $\beta\in \mathbb K$ with $h_{\mathbb X}(\beta)>C$ 
it holds that:
\begin{equation}\label{our_roth}
\sum^n_{i=1}\int_{S_i}\log^-\vert \beta-\alpha_i\vert_\omega d\mu(\omega)>-(2+\varepsilon)h_\mathbb X(\beta)\,.
\end{equation}
\end{namedtheorem}
We will show that Theorem \ref{CorCorv} is a consequence of Theorem \bref{thm:main}{(A)}, but on the other hand Vojta's inequality is stronger and moreover it doesn't depend on the fixed partition of $S$. Nevertheless, we get a generalisation of Theorem \ref{vojta}  under the following hypotheses:
\begin{itemize}
\item We slightly relax the  strong $\mu$-equicontinuity condition and we allow the existence of arbitrary small sets where the equicontinuity fails (see Definition \ref{muequi}).
\item   We assume that  the functions $\omega\to \log^-|\beta|_\omega$ are \emph{uniformly integrable} while $\beta$ varies (see definition \ref{un_int})
\end{itemize}

Our second main theorem is then the following:

\begin{namedtheorem}[\hypertarget{thm:(B)}{Theorem (B)}]
Let $\mathbb X=(\mathbb K,\Omega, \phi)$ be a proper adelic curve satisfying the $\mu$-equicontinuity condition and the uniform integrability condition. Fix some distinct elements $\alpha_1,\dots ,\alpha_n\in\mathbb K$. Let $S$ be a measurable  subset of $\Omega$ of finite measure. Then for any $\varepsilon>0$ and any $c\in\mathbb R$ there exists a real constant $C>0$ (depending on the fixed data) such that for any $\beta\in \mathbb K$  with  $h_{\mathbb X}(\beta)>C$   
it holds that:
\begin{equation}\label{our_roth_encore}
\int_{S}\,\min_{1\le i\le n}\left(\log^-\vert \beta-\alpha_i\vert_\omega\right) d\mu(\omega)>-(2+\varepsilon)h_\mathbb X(\beta)+c
\end{equation}
\end{namedtheorem}
Finally, in  Example \ref{Qbar} we show that  the $\mu$-equicontinuity condition doesn't hold for all proper adelic curves; in our example we consider $\overline{\mathbb Q}$ endowed with a natural structure of adelic curve naturally inherited from $\mathbb Q$.

Our strategy consists  in weaving together the ideas of Corvaja and Vojta to get  rather elementary proofs which are independent from Arakelov geometry. On the other hand we don't have, at the moment, any examples of interesting proper adelic curve different from the ones already known. It is a much harder problem to find new concrete cases in which Roth's theorem holds. After all, a highly nontrivial  achievement of Vojta's work consists in showing that the $\mu$-equicontinuity   and the uniform integrability hold  for arithmetic function fields. He does this by using the geometry  of complex fibres at infinity  appearing in Arakelov geometry.

 The very coarse overview of the proofs is the following: Roth's theorem is about the simultaneous approximation of some distinct elements $\alpha_1,\ldots, \alpha_n$ that by simplicity (in fact it will be enough to consider this case) can be fixed in $\mathbb K$, with an element $\beta\in\mathbb K$. The existence of a very special ``interpolating polynomial'' $\delta$ for such elements (section \ref{aux}) allows us to write some  integral estimates  for  measurable functions  on $\Omega$  satisfying some technical properties related to the heights of the $\alpha_i$'s and $\beta$ (section \ref{estimates}). Then, assuming  that Roth's theorem is false  leads to the construction  of a measurable function $\theta:S\to\mathbb R_{\ge 0}$ that gives the desired contradiction on the integral estimates previously found. The crucial point of the proof consists in the construction of $\theta$, and this is exactly where we need the additional technical conditions on the adelic curve.

We finally mention that section \ref{introad} is a brief review of the theory of adelic curves introduced  in \cite{Ch-Mo}, and moreover that Appendix \ref{ap} sketches the construction of the interpolating polynomial employed in \cite{C}.

\paragraph{Acknowledgements.} Both the authors want to express their gratitude to \emph{Pietro Corvaja} for introducing them to the subject and for his enlightening explanations. They also want to thank \emph{Lorenzo Freddi} and \emph{Huyai Chen} for replying to several questions.

A special mention goes to the anonymous referee for their incredible peer-reviewing work with this article. Their deep understanding of the topic and their compelling comments  were extremely useful to fix some mistakes.

The first author was partially supported by the EPSRC grant EP/M024830/1 ``Symmetries and correspondences: intra-disciplinary developments and applications'', partially by the research grant ``Higher Arakelov geometry, Cryptography and Isogenies of Elliptic Curves'' conferred by the University of Udine, and partially by the Ben-Gurion university of the Negev.

The second author wants to thank his wife \emph{Elena}.

\section{Adelic curves}\label{introad}
We will use the following notations throughout the whole paper:

$$\log^+ x := \max\{0, \log x\}\,,\quad \log^- x := \min\{0, \log x\}\,;\quad\forall x\in \mathbb R_{>0}$$

In this section we closely follow \cite[Chapter 3]{Ch-Mo}. For simplicity we restrict to the characteristic $0$ case, but all the definitions work also in positive characteristic.
\begin{defn}\label{AdCu}
 Let $\mathbb K$ be a field of characteristic $0$,  let $M_{\mathbb K}$ be the set of all absolute values of $\mathbb K$ and let $\Omega=(\Omega,\mathcal A, \mu)$ be a  measure space endowed with  a map 
\begin{eqnarray*}
\phi\colon \Omega&\to& M_{\mathbb K}\\
\omega &\mapsto & |\cdot|_\omega:=\phi(\omega)\,.
\end{eqnarray*} 
 such that for any $a\in \mathbb K^\times$, the real valued function $\omega\mapsto \log |a|_\omega$ lies in  $L^1(\Omega,\mu)$.  The triple $\mathbb X=(\mathbb K,\Omega,\phi)$ is called an \emph{adelic curve}; $\Omega$ and $\phi$ are respectively the \emph{parameter space} and the  \emph{parametrization}.  We denote with $\Omega_{\infty}$ the subset of $\Omega$ made  of all elements $\omega$ such that $|\cdot|_{\omega}$ is archimedean. We set  $\Omega_0:=\Omega\setminus \Omega_{\infty}$.
\end{defn}
\begin{rem}
We also recall that a more general notion of adelic curve, with the requirement that $|\cdot|_\omega$ is an absolute value only almost everywhere for $\omega\in\Omega$, had been already given in  \cite{gu}  under the name of $M$-field. 
\end{rem}

It is easy to show that the set $\Omega_{\infty}$ is always measurable \cite[Proposition 3.1.1]{Ch-Mo}.

\begin{defn}
An adelic curve $\mathbb X=(\mathbb K,\Omega,\phi)$ is said to be \emph{proper}  if for any $a\in \mathbb K^\times$:

\begin{equation}\label{pr_for}
\int_{\Omega} \log |a|_\omega\, d\mu(\omega)=0\,.
\end{equation}
\end{defn}

Let's see  examples of adelic curves:

\begin{expl}\label{counting}
Any field $(k,\mathcal V_k)$ satisfying the product formula in the sense of \cite{C} is a proper adelic curve. In fact in this case $\Omega=\mathcal V_k$, $\phi$ is the identity and  $\mu$ is the counting measure.
\end{expl}

\begin{expl}
An arithmetic function field  $K$ with a big polarisation is a proper  adelic curve. A quick description of this has been already given in section \ref{hi}. For details see \cite[Section 3]{vo}. 
\end{expl}
\begin{expl} 
A polarised algebraic function field (in $d\ge 1$ variables) over a field of characteristic $0$ can be endowed with a structure of proper adelic curve. For details see  \cite[3.2.4]{Ch-Mo}. 
\end{expl}
In the remaining part of this section we study the behaviour of adelic curves with respect to field extensions. In particular  let's  fix an adelic curve $\mathbb X=(\mathbb K,\Omega_{\mathbb K},\phi_{\mathbb K})$, and let $\mathbb L$ be a \emph{finite} extension of $\mathbb K$, our goal is to endow it with a canonical structure of adelic curve  ``coming from $\mathbb K$''. In other words, we have to define a parameter space $\Omega_{\mathbb L}$ and a parametrization $\phi_{\mathbb L}$ in a canonical way by starting from $\Omega_{\mathbb K}$ and $\phi_{\mathbb K}$. For any $\omega\in\Omega_{\mathbb K}$ we denote with $M_{\mathbb L,\omega}$ the set of absolute values of $\mathbb L$ extending $|\cdot|_{\omega}$, so we put:
$$\Omega_{\mathbb L}:=\bigsqcup_{\omega\in\Omega_{\mathbb K}}M_{\mathbb L,\omega}$$
and we have a natural projection map $\pi_{\mathbb L|\mathbb K}:\Omega_{\mathbb L}\to \Omega_{\mathbb K}$ whose fibres are $M_{\mathbb L,\omega}$, for any $\omega$. The inclusion $M_{\mathbb L,\omega}\subset M_{\mathbb L}$ clearly induces a parametrization $\phi_{\mathbb L}\colon\Omega_{\mathbb L}\to M_{\mathbb L}$ and for any $\nu\in\Omega_{\mathbb L}$ we put $|\cdot|_\nu:=\phi_{\mathbb L}(\nu)$. We obtain the following commutative diagram:

\begin{equation}\label{sqdiag}
\begin{tikzcd}
 \Omega_{\mathbb L}\arrow[r,"\pi_{\mathbb L|\mathbb K}"] \arrow[d,"\phi_{\mathbb L}"] & \Omega_{\mathbb K}\arrow[d,"\phi_{\mathbb K}"]\\
M_{\mathbb L}\arrow[r] & M_{\mathbb K}
\end{tikzcd}
\end{equation}
where the bottom map is the restriction function. Note that $\Omega_{\mathbb L}$ can be abstractly defined as the fibered product in the category of sets. Now on $\Omega_{\mathbb L}$ we put the $\sigma$-algebra $\mathcal B$ generated by $\pi_{\mathbb L|\mathbb K}$ and all the real maps $\Omega_{\mathbb L}\ni\nu\mapsto|a|_{\nu}$, for any $a\in\mathbb L^\times$ (on $\mathbb R$ we put the standard Lebesgue measure). We want to define a suitable measure $\eta$ on the measurable space $(\Omega_{\mathbb L},\mathcal B)$. This requires a bit of work, since in general there is no straightforward definition of pullback measure through a measurable map. Nevertheless, in the  case of our $\pi_{\mathbb L|\mathbb K}$ we explain how it is possible to define the pullback $\eta=\pi^\ast_{\mathbb L|\mathbb K}\mu$ (actually this is a  construction from measure theory which works in full generality any time we have a measure fiberwise).  Consider a fiber $M_{\mathbb L,\omega}\subset\Omega_{\mathbb L}$, then for any $\nu\in M_{\mathbb L,\omega}$ we can put
\begin{equation}\label{prob}
P_\omega(\nu):=\frac{[\mathbb L_\nu:\mathbb K_\omega]}{[\mathbb L:\mathbb K]}
\end{equation}
where $\mathbb L_\nu$ and $\mathbb K_\omega$ denote the completions with respect to $|\cdot|_\nu$ and $|\cdot|_\omega$ respectively. Thanks to the well known equality $\sum_{\nu\in M_{\mathbb L,\omega}}[\mathbb L_\nu:\mathbb K_\omega]=[\mathbb L:\mathbb K]$ (See \cite[Ch. II, Corollary 8.4]{Neu}), we conclude that equation (\ref{prob}) induces a probability measure on the fibre $M_{\mathbb L,\omega}$, with respect to the power set. Now, for any function $f:\Omega_{\mathbb L}\to\mathbb R$, by using the fiberwise integral along each probabilised fiber $M_{\mathbb L,\omega}$, we define the map $I_{\mathbb L|\mathbb K}(f):\Omega_\mathbb K\to\mathbb R$ as:
$$I_{\mathbb L|\mathbb K}(f):\omega\mapsto \int_{M_{\mathbb L,\omega}} \!\!f\,dP_\omega=\!\sum_{\nu\in M_{\mathbb L,\omega}}\!P_\omega(\nu)f(\nu)$$
\begin{prop}\label{fiber_int}
The linear operator $I_{\mathbb L|\mathbb K}$ sends $\mathcal B$-measurable functions to $\mathcal A$-measurable functions.
\end{prop}
\proof
See \cite[Theorem 3.3.4]{Ch-Mo}.
\endproof
At this point we are ready to define the measure $\eta$. For any $E\in\mathcal B$ we put:
\begin{equation}\label{pullback_int}
\eta(E):=\int_{\Omega_\mathbb K} I_{\mathbb L|\mathbb K}(\chi_{E})\, d\mu
\end{equation}
where $\chi_{E}$ is the characteristic function of $E$. Note that the integral of equation (\ref{pullback_int}) makes sense because of Proposition \ref{fiber_int}. 

\begin{thm}\label{pullback_mea}
The following statements hold:
\begin{itemize}
\item[$(1)$] The map $\eta$ defined above is a measure on $(\Omega_{\mathbb L},\mathcal B)$ such that for any $\mathcal B$-measurable function $f$ we have
$$\int_{\Omega_\mathbb L}fd\,\eta=\int_{\Omega_\mathbb K}I_{\mathbb L|\mathbb K}(f)\, d\mu\,.$$
\item[$(2)$] $f\in L^1(\eta)$ if and only if $I_{\mathbb L|\mathbb K}(|f|)\in L^1(\mu)$.
\item[$(3)$] The pushforward measure of $\eta$ through $\pi_{\mathbb L|\mathbb K}$ is $\mu$.
\item[$(4)$] With the above constructions the triple $\mathbb Y=(\mathbb L,\Omega_{\mathbb L},\phi_{\mathbb L})$ is an adelic curve. Moreover for any $b\in\mathbb L^\times$
\begin{equation}\label{extproper}
[\mathbb L:\mathbb K]\int_{\Omega_\mathbb L}\log|b|_\nu d\eta(\nu)=\int_{\Omega_\mathbb K}\log|N_{\mathbb L|\mathbb K}(b)|_\omega d\mu(\omega)
\end{equation}
and in particular if $\mathbb X$ is proper, then also $\mathbb Y$ is proper.

\end{itemize}
\end{thm}
\proof
See \cite[Theorem 3.3.7]{Ch-Mo}.
\endproof
At this point we study  \emph{algebraic} extensions of adelic curves. Let's fix the adelic curve $\mathbb X=(\mathbb K,\Omega_K,\phi_K)$ and let $\mathbb L$ an algebraic extension of $\mathbb K$. We denote with $\mathcal F_{\mathbb L|\mathbb K}$ the family of finite field extensions on $\mathbb K$ contained in $\mathbb L$. Clearly $\mathcal F_{\mathbb L|\mathbb K}$ is a directed set with respect to the inclusion, and for any $\mathbb K',\mathbb K''\in\mathcal F_{\mathbb L|\mathbb K}$  such that  $\mathbb K'\subseteq\mathbb K''$ we have a morphism of measurable spaces 
$$\pi_{\mathbb K''|\mathbb K'}\colon (\Omega_{\mathbb K''}, \mathcal B'')\to (\Omega_{\mathbb K'}, \mathcal B')$$
and an operator $I_{\mathbb K''|\mathbb K'}$ sending integrable functions to integrable functions as described above in the case of finite extensions. In other words we obtain an inverse system of measure spaces, and we would like to define the adelic structure on $\mathbb L$ as ``a projective limit''. Unfortunately, in the category of measure spaces we don't have a straightforward notion of projective limit, therefore we need again a bit of extra work.  We can define respectively $M_{\mathbb L,\omega}$, $\Omega_{\mathbb L}$ and $\phi_{\mathbb L}\colon \Omega_{\mathbb L}\to M_{\mathbb L}$ exactly as we did before in the case of finite extensions, but we need to construct an adequate structure of measure space on $\Omega_{\mathbb L}$. For any $K'\in\mathcal F_{\mathbb L|\mathbb K}$ we have  a map $\pi_{\mathbb L|\mathbb K'}:\Omega_{\mathbb L}\to \Omega_{\mathbb K'}$ and a square diagram like (\ref{sqdiag}).  It turns out that $\pi_{\mathbb L|\mathbb K'}$ is surjective \cite[Proposition 3.4.5]{Ch-Mo}. We endow $\Omega_{\mathbb L}$ with the $\sigma$-algebra $\Sigma$ generated by the maps $\{\pi_{\mathbb L|\mathbb K'}\}_{\mathbb K'\in \mathcal F_{\mathbb L|\mathbb K}}$, and it can be shown that $(\Omega_{\mathbb L},\Sigma)$ is the projective limit of the inverse system $\{(\Omega_{\mathbb K'},\mathcal B')\}_{\mathbb K'\in \mathcal F_{\mathbb L|\mathbb K}}$ in the category of measurable spaces. It remains the issue of putting a canonical measure $\lambda$ on $(\Omega_{\mathbb L},\Sigma)$. This process is quite technical, but similarly to the case of finite extensions, it can be done by using a fiberwise integration on each $M_{L,\omega}$; all the details are given in \cite[3.4.2]{Ch-Mo}. What we really need is the fact that we can construct an adelic curve $(\mathbb L,\Omega_{\mathbb L},\lambda)$ which is proper if $\mathbb X$ is proper and such that for any $f\in L^1(\mu)$ we have that $f\circ \pi_{\mathbb L|\mathbb K}\in L^1(\lambda)$ with
\begin{equation}\label{int_ext}
\int_{\Omega_{\mathbb L}}(f\circ \pi_{\mathbb L|\mathbb K})d\lambda=\int_{\Omega_{\mathbb K}} fd\mu\,.
\end{equation}
Below we give the notion of height for proper adelic curves:
\begin{defn}\label{height}
Let $\mathbb X=(\mathbb K, \Omega,\phi)$ be a proper adelic curve and let $\overline {\mathbb K}$ be an algebraic closure of $\mathbb K$. Then we have  a proper adelic curve $\overline{\mathbb X}=(\overline {\mathbb K},\overline{\Omega},\overline\phi)$ and the \emph{(naive) height} of an element $a\in\overline {\mathbb K}^\times$ is defined as:
$$
h_{\mathbb X}(a):=\int_{\overline{\Omega}} \log^+|a|_{\nu}\, d\chi(\nu)\,.
$$
where $\nu$ denotes a generic element of $\overline \Omega$ and $\chi$ is the measure on $\overline \Omega$. Moreover we set $H_{\mathbb X}:=e^{h_{\mathbb X}}$. 
\end{defn}
From now on we always assume that for an adelic curve $\mathbb X=(\mathbb K, \Omega,\phi)$ we have fixed algebraic closure of $\mathbb K$, therefore also $\overline{\mathbb X}$ is fixed and we use the same notations of Definition \ref{height}.
If for $\nu\in\overline\Omega$, $\vert\cdot\vert_\nu$ is an archimedean absolute value, then by Ostrowski's theorem we know that there exists a real number $\varepsilon(\nu)\in\,]0,1]$ such that $\vert\cdot\vert_\nu=\vert\cdot\vert^{\varepsilon(\nu)}$ where on the right we mean the standard euclidean absolute value on $\mathbb R$ or $\mathbb C$. Thus we have a map $\varepsilon:\overline\Omega_{\infty}\to ]0,1]$ which can be extended to  $\varepsilon:\overline\Omega\to [0,1]$ by putting $\varepsilon_{|\overline\Omega_0}:=0$. For instance,  for an archimedean $\vert\cdot\vert_\nu$ we have $\log \vert 2\vert_\nu=\varepsilon(\nu)\log 2$, therefore we obtain the explicit expression of the function  $\varepsilon$ on the whole $\overline\Omega$:
$$\varepsilon(\nu)=\frac{\log^+\vert2\vert_\nu}{\log 2}\,.$$
Clearly $\varepsilon(\nu)$ is a measurable function. We can always take a  scaling  $\mu'$ of the measure $\mu$ on $\Omega$
so that  get a new height $h'_{\mathbb X}$ that  satisfies $h'_{\mathbb X}(2)\le\log 2$.  Notice that if $\mathbb X$ is proper, then it remains proper after any scaling of the measure $\mu$.  From now on, when we are given an adelic curve $\mathbb X=(\mathbb K, \Omega,\phi)$, we can always assume that we have performed the above mentioned scaling  of the measure $\mu$ on $\Omega$ so that $h_{\mathbb X}(2)\le \log 2$.
\begin{defn}
Let $P(X_1,\cdots ,X_N)$ be a polynomial over $\overline{\mathbb K}$, $\bm{\alpha}\in \overline{\mathbb K}^N$ and $\bm i=(i_1,\cdots ,i_N)\in \mathbb N^N$. We set:
$$
\Delta^{\bm i}P(\bm{\alpha}):=\frac{1}{i_1!i_2!\cdots i_N!}\,\frac{\partial^{i_1+i_2+\ldots+i_N}P}{\partial X^{i_1}_1 \partial X^{i_2}_2\ldots \partial X^{i_N}_N}(\bm{\alpha}).
$$
We can define the \emph{local height of $P$ at $\nu\in\overline\Omega$} in the following way:
$$h_{\nu}(P):=\log\left(\max_{\bm i\in\mathbb N^N}\left\{\left |\Delta^{\bm i} P(0,\ldots,0)\right|_{\nu}\right\}\right)$$
and then we have also the notion of \emph{global height} of $P$:
$$h_{\mathbb X}(P):=\int_{\overline\Omega} h_{\nu}(P) d\chi(\nu)$$
We put $H_{\mathbb X}(P):=e^{h_{\mathbb X}(P)}$.
\end{defn}

The following  estimates of $\log\vert\Delta^{\bm i}P(\alpha)\vert_\omega$ in terms of the local height of $P$ will be very useful later:

\begin{lem}\label{corvajalemmepag166} Let $(\mathbb K, \vert\cdot\vert_\omega)$ be a field with an absolute value. Let $P\in\mathbb K[X_1,\dots X_N]$ such that $N\geq 1$. Then for every $\bm{\alpha}\in\mathbb K^N$
$${\log}\vert \Delta^{{\bm i}}P(\bm{\alpha})\vert_\omega\leq 
h_\omega(P)+\log^+\bigg\lvert\prod^N_{j=1}\left(1+\deg_{X_j} P\right)\bigg\rvert_\omega+\sum_{j=1}^{N}(\log^+\vert 2\vert_\omega+{\log}^+\vert\alpha^{(j)} \vert_\omega){\deg}_{X_j}P
$$
\end{lem}
\begin{proof} See  \cite[Lemme, page 166]{C}.
\end{proof}

We conclude the section with some rather straightforward results about heights. First  of all when we want to calculate the heights of elements lying in $\mathbb K^\times$, we don't need to involve the algebraic closure $\overline{\mathbb K}$ in the integrals:
\begin{prop} Let $\mathbb X=(\mathbb K, \Omega,\phi)$ be an adelic curve. If $a\in\mathbb K^\times$, then:
$$h_{\mathbb X}(a)=\int_{\Omega_\mathbb{K}} \log^+|a|_{\omega} d\mu(\omega)\,.$$
Moreover, the same result holds for the heights of polynomials in $\mathbb K[X_1,\ldots, X_N]$.
\end{prop} 
\proof
It is an immediate consequence of equation (\ref{int_ext}).
\endproof
In order to simplify the notations, we often omit the subscript $\mathbb X$ attached to the heights when the adelic curve is fixed and there is no confusion.
\begin{prop}\label{properties}
The height function of a proper adelic curve $(\mathbb X, \Omega,\phi)$ satisfies the following properties for any $a,b,a_1,\ldots,a_m\in\mathbb K$ and  any measurable set $S\subseteq\Omega$
\begin{itemize}
\item[$(1)$] $h(a)=h(a^{-1})$
\item[$(2)$]  $-h(a)\le\int_S\log\vert a\vert_\omega d\mu(\omega)\le h(a)$
\item[$(3)$] $\int_S\log^-\vert a\vert_\omega d\mu(\omega)\ge -h(a)$
\item[$(4)$]  $h(a_1+\ldots +a_m)\le h(m)+h(a_1)+\ldots+h(a_m)$
\item[$(5)$]  $\int_S\log \vert a-b\vert_\omega d\mu(\omega)\ge -\log 2- h(a)-h(b)$
\item[$(6)$]$\int_S\log^- \vert a-b\vert_\omega d\mu(\omega)\ge -\log 2- h(a)-h(b)$
\end{itemize}
\end{prop}
\proof $(1)$ It follows from the product formula and from the fact that $\log\vert a\vert_\omega=\log^+\vert a\vert_\omega-\log^+\vert \frac{1}{a}\vert_\omega$.

 $(2)$ By definition $\int_S\log\vert a\vert_\omega d\mu(\omega)\le h(a)$, so for the other inequality it is enough to use $(1)$.
 
 $(3)$ We use the equality $\log^-\vert a\vert_\omega=-\log^+\vert a^{-1}\vert_\omega$ and we obtain:
\[
\int_S\log^-\vert a\vert_\omega d\mu(\omega)=-\int_S\log^+\vert a^{-1}\vert_\omega d\mu(\omega) \ge -h(a^{-1})=-h(a)
\]

$(4)$ It follows from $\vert a_1+\ldots+a_m \vert_{\omega}\le m \max_i \vert a_i\vert_{\omega}$.
 
 $(5)$ and $(6)$ are  direct consequences of $(2)$-$(4)$ and the fact that $h(2)\le\log(2)$.
\endproof

Here we stress  that the entries $(5)$-$(6)$ of Proposition \ref{properties} replace the classical Liouville inequality  for heights. Finally we recall an important property of heights: 

\begin{defn}
A proper adelic curve $\mathbb X=(\mathbb K,\Omega, \mu)$ satisfies the \emph{Northcott property} if for any $C\in\mathbb R$ the set $\{\alpha\in\mathbb K\colon h_{\mathbb X}(a)\}\le C$ is finite.
\end{defn}

Arithmetic function fields satisfy Northcott properties thanks to \cite[Theorem 4.3]{Mor}.

\section{The interpolating polynomial}\label{aux}

We fix a proper adelic curve $\mathbb X=(\mathbb K,\Omega, \phi)$  and an algebraic closure of $\mathbb K$. In this section we recall the existence of an interpolating polynomial $\delta\in\mathbb K[X_1,\ldots,X_N]$ associated to  some elements $\bm{\alpha}_1,\ldots,\bm{\alpha}_n\in\mathbb K^N$ having some explicit bounds on: the degree, the $\bm{d}$-index at all the $\bm{\alpha}_j$'s and the height. The complete construction of $\delta$ can be found in \cite{C}, and we will recall it in appendix \ref{ap}.

We fix for the whole section the following data: two natural numbers $n,N\ge 2$, and  a vector $\bm{d}=(d_1,\ldots ,d_N)\in(\mathbb R_{>0})^N$.  We say that two vectors  $\bm{\alpha}=(\alpha^{(1)},\ldots,\alpha^{(N)})$, $\bm{\beta}=(\beta^{(1)},\ldots,\beta^{(N)})$  are \emph{componentwise different} if $\alpha^{(j)}\neq\beta^{(j)}$ for  $j=1,\ldots, N$.
\begin{defn}  The \emph{$\bm{d}$-index of $P(X_1,\dots ,X_N)\neq 0$ at $\bm{\alpha}\in\mathbb K^N$} is the real number:
$$
\Ind_{\bm{\alpha},\bm{d}}(P):=\min_{\bm i\in \mathbb N^N}\left\{\sum_{j=1}^{N}\frac{i_j}{d_j}\in\mathbb R\colon \Delta^{\bm i}P(\bm{\alpha})\neq 0\right \}
$$
\end{defn}
Let's fix  $t\in\mathbb R$ such that $0<t<N$, the following two sets will play a central role in the theory:
$$
\sG_t:=\left \{\bm i\in \mathbb N^N\colon i_j\leq d_j\; \forall j=1,\dots ,N,\, {\rm{and}}\sum_{j=1}^{N}\frac{i_j}{d_j}\leq t\right \},
$$
$$
\sC_t:=\left\{(x_1,\dots ,x_N)\in [0,1]^N \colon \sum_{j=1}^{N}x_j \leq t \right \}
$$
The Lebesgue measure of $\sC_t$ will be denoted as $V(t)$, and for simplicity of terminology we will refer to it simply as ``volume''.

\begin{lem}\label{Lemmepagina156}  The cardinality of  $\sG_t$  is asymptotic to $d_1 d_2\ldots d_N V(t)$.
\end{lem}
\begin{proof}See \cite[p. 157]{bogu}.
\end{proof}

Now fix some vectors $\bm{\alpha}_1,\ldots,\bm{\alpha}_n\in\mathbb K^N$ where $\bm{\alpha}_h=(\alpha^{(1)}_h,\ldots,\alpha^{(N)}_h)$ for every $h=1,\ldots, n$ and let  $\bm{X}=(X_1,\ldots,X_N)$ be a vector of variables. For any two multi-indices  $\bm {a}=(a_1,a_2,\dots ,a_N)$ and $\bm {i}=(i_{1},i_{2},\dots ,i_{N})$ of $\mathbb N^{N}$ we use the following notations:
$$
{\bm a\choose \bm i}:={{{a_1}}\choose {i_1}}  {{{a_2}} \choose {i_2}} \ldots {{{a_N}}\choose {i_N}}
$$
$$
\bm{\alpha}_h^{\bm i}:=(\alpha_{h}^{(1)}){^{{i_1}}}  ({\alpha}_{h}^{(2)}){^{{i_2}}}\ldots (\alpha_{h}^{(N)}){^{{i_N}}}
$$
$$
\bm{X}^{\bm i}:=X^{i_1}_1X^{i_2}_2\ldots X^{i_N}_N
$$
with the convention ${{{p}}\choose {q}}=0$ if $0\leq p<q$ for the binomial coefficient. Now, consider $\gamma\in\mathbb R$ such that $0<\gamma<\frac{1}{2nN^2}$; we always assume that $\frac{d_{j+1}}{d_j}\le\gamma$ for any $j=1,\ldots, N-1$ , which means in particular that  $d_j=O(d_1)$ for any $j$. We also put $\eta:=2\gamma n < \frac{1}{N^2}$ and $d:=d_1d_2\ldots d_N$.

We have the following result about the existence of a polynomial $\delta(X_1,\ldots,X_N)\in\mathbb K[X_1,\ldots X_N]$ with some prescribed properties. In the appendix A we will sketch Corvaja's construction of $\delta(X_1,\ldots,X_N)$ adapting it to the case of adelic curves:
\begin{prop}\label{interpolationpoly} Assume that in the adelic curve $\mathbb X$ the condition $h_{\mathbb X}(2)\le \log 2$ is satisfied and assume that the number $\eta\in\mathbb R$ is chosen as explained above.   Let's fix some vectors $\bm{\alpha}_1,\ldots,\bm{\alpha}_n,\bm{\beta}\in\mathbb K^N$ that are pairwise componentwise different.  Moreover let's choose some parameters $s, t_1,\dots, t_n\in\mathbb R$, with $0<s<1$ and $0< t_h<\frac{N}{2}$ for $h=1,\dots ,n$ such that the following condition on volumes is verified:
\begin{equation}\label{basicinequality}
(1+\eta)^N<V(s)+\sum_{h=1}^{n}V(t_h)<1+2N\eta
\end{equation}
Then there exists a polynomial $\delta(X_1,\ldots X_N)\in\mathbb K[X_1,\ldots, X_N]$ satisfying the following properties:
\begin{enumerate}
\item[$(1)$] $\delta(\bm\beta)\neq 0$;
\item[$(2)$] ${\deg}_{X_{j}}\delta\leq dd_jV(s)+O(d)$,  for $1\leq j\leq N$;
\item[$(3)$] ${\rm{Ind}}_{\bm{\alpha}_h,\bm{d}}(\delta)\geq dV(s)\left ( t_h-s-\frac{2N^2\eta}{V(s)} \right)+ O(d^{N-1}_1) $,  for $1\leq h\leq n$;
\item[$(4)$] $\displaystyle h_{\mathbb X}(\delta)\leq d\sum_{j=1}^{N}d_j\left( \log 2+   \sum_{h=1}^{n}V(t_h)h(\alpha^{(j)}_{h})\right)+O(d\log d)$;
\end{enumerate}
\end{prop}
\proof
See \cite[Proposition 2.6]{C}.
\endproof

\section{Integral estimates}\label{estimates}
This technical section is ``the heart'' of the proof of our results since here we will prove some integral bounds for very particular integrable functions $\theta:S\subset\Omega\to \mathbb R_{\ge 0}$. We continue with all the notations fixed in section \ref{aux} since we want to make full use of Proposition \ref{interpolationpoly}. 

Consider an adelic curve $(\mathbb K,\Omega, \phi)$ and a set of vectors $\bm{\alpha}_1,\ldots,\bm{\alpha}_n,\bm{\beta}\in\mathbb K^N$ which are componentwise different. We construct the following  matrices $T:=T(\bm{\alpha}_1,\ldots\bm{\alpha}_n)\in M(n\times N,\mathbb K)$ and $T(\bm{\beta})\in M(n+1\times N,\mathbb K)$:

\[
T:=\begin{pmatrix}
 \alpha^{(1)}_1& \alpha^{(2)}_1  &\ldots &\alpha^{(N)}_1 \\
 \alpha^{(1)}_2& \alpha^{(2)}_2  &\ldots &\alpha^{(N)}_2\\
\vdots & \vdots & \vdots & \vdots\\
\alpha^{(1)}_n& \alpha^{(2)}_n  &\ldots &\alpha^{(N)}_n
\end{pmatrix};\quad
T(\bm{\beta})=\begin{pmatrix}
 \alpha^{(1)}_1& \alpha^{(2)}_1  &\ldots &\alpha^{(N)}_1 \\
 \alpha^{(1)}_2& \alpha^{(2)}_2  &\ldots &\alpha^{(N)}_2\\
\vdots & \vdots  & \vdots  &\vdots\\
\alpha^{(1)}_n& \alpha^{(2)}_n  &\ldots &\alpha^{(N)}_n\\
\beta^{(1)}& \beta^{(2)}  &\ldots &\beta^{(N)}
\end{pmatrix}
\]
We denote by $\bm{\alpha}^{(j)}$, $j=1,\ldots, N$ the columns of $T$, that is:
\begin{equation*}
\bm{\alpha}^{(j)}=\begin{pmatrix}
 \alpha^{(j)}_1 \\
 \alpha^{(j)}_2\\
\vdots \\
\alpha^{(j)}_n
\end{pmatrix}
\end{equation*} 
and by $\bm{\alpha}_{h}$, for  $h\in\{1,2,\dots ,n\}$, the rows:
\begin{equation*}
\bm{\alpha}_{h}=\begin{pmatrix}
 \mathbb \alpha^{(1)}_h, \alpha^{(2)}_h,\ldots, \alpha^{(N)}_h
 \end{pmatrix}
\end{equation*}
Note that we are asking for the matrices $T$ and $T(\bm{\beta})$ to have componentwise different rows. Now we need to define a list of properties depending on the aforementioned matrices:
\begin{defn}\label{boxing}
For the matrix $T(\bm{\beta})$, consider the following real numbers for any $j=1,\ldots, N$: 
\begin{equation}\label{pivot}
\rho_{j}:=4^{2N!}H(\beta^{(j)})\prod_{h=1}^{n}H(\alpha_h^{(j)})^{\frac{2N!}{n}}
\end{equation}
\begin{equation}\label{pivotdue}
\rho'_{j}:=4^{N!}H(\beta^{(j)})\prod_{h=1}^{n}H(\alpha_h^{(j)})^{\frac{N!}{2n}}
\end{equation}
We say that $T(\bm{\beta})$ satisfies the \emph{$h$-gap property} if the following inequality is satisfied:
$$
\frac{{\rm{log}}\, \rho_j }{{\rm{log}}\,\rho'_{j+1}}<\frac{1}{4nN^2N!},\qquad \forall j=1,\ldots, N-1
$$
\end{defn}

\begin{defn}\label{bounding} Fix a measurable set $S=S_1\sqcup\ldots\sqcup S_n$. We say that an integrable function $\theta\colon S\to\mathbb R_{\ge 0}$ is a \emph{column bounding function for} $T(\bm{\beta})$ on $S$ if 
the following inequality holds:
\begin{equation}\label{boundingbis}
-\frac{1}{\log \rho_{j}}{\log}\vert \alpha^{(j)}_h-\beta^{(j)}\vert_\omega\ge\theta (\omega)\qquad \forall j=1,\ldots,N\,,\; \forall h=1,\ldots,n\,,\; \forall\omega\in S_h 
\end{equation}
\end{defn}

\begin{defn}
Fix a measurable set $S=S_1\sqcup\ldots\sqcup S_n$. For any column $\bm{\alpha}^{(j)}\in\mathbb K^N$  of the matrix $T$ and any $b\in\mathbb K$ we define the following quantity:
$$
\lambda(\bm{\alpha}^{(j)},b):=\frac{1}{V(s)}{\rm{log}}4+h(b)+\frac{1}{V(s)}\sum_{h=1}^{n}V(t_h)h(\alpha_{h}^{(j)})
$$
\end{defn}

\begin{defn}\label{lambda_gap} We say that  $T(\bm{\beta})$ satisfies the \emph{$\lambda$-gap property} of width $\eta\in\mathbb R_+$ if
$$
\max_{1\leq j\leq N-1}\frac{\lambda(\bm{\alpha}^{(j)},\beta^{(j)})}{\lambda(\bm{\alpha}^{(j+1)},\beta^{(j+1)})}<\frac{\eta}{2n}
$$
\end{defn}

\begin{defn}\label{lambda_bou} Fix a matrix $T(\bm{\beta})$. An integrable function $\theta\colon S\to \mathbb R_{\ge 0}$ is \emph{$\lambda$-bounding} if:
\begin{equation}\label{pequality}
\frac{-1}{\lambda(\bm{\alpha}^{(j)},\beta^{(j)})}
{\log}\vert \alpha^{(j)}_h-\beta^{(j)}|_{\omega} \ge\theta(\omega)\qquad \forall j=1,\dots ,N;\;\forall h=1,\dots,n;\;\forall \omega\in S_h
\end{equation}
\end{defn}

\begin{rem}\label{rem_forthelemma}
Note that if $\theta$ is $\lambda$-bounding or column bounding then it follows immediately that $\vert \alpha^{(j)}_h-\beta^{(j)}|_{\omega}\le 1$ for any  $j=1,\dots ,N$, $h=1,\dots,n$, $\omega\in S_h$.
\end{rem}

 The following easy lemma will be useful:
\begin{lem}\label{log2}
Let $x,y\in \mathbb K$ and let $\omega$ such that $\vert x-y\vert_\omega\le 1$. Then   $\log^+\vert x\vert_\omega\le \log^+\vert 2\vert_\omega+\log^+\vert y\vert_{\omega}$.
\end{lem}
\proof
 We distinguish two cases:\\
\emph{$\vert\cdot\vert_\omega$ is  ultrametric.} Then  
$$
\vert x\vert_\omega=\vert x-y+y\vert_\omega\le \max\left\{ \vert x-y\vert_{\omega}, \vert y\vert_{\omega}\right\}\le \max\left\{1, \vert y\vert_{\omega}\right\}\,.
$$
If $\vert y\vert_{\omega}\ge 1$, then clearly $\log^+\vert x\vert_\omega\le     \log^+\vert y\vert_\omega$. If $\vert y\vert_{\omega}<1$, then $\vert x\vert_{\omega}\le 1$ which means  $0=\log^+\vert x\vert_{\omega}\le \log^+\vert 2\vert_\omega+\log^+\vert y\vert_{\omega}$.

\emph{$\vert\cdot\vert_\omega$ is  archimedean}. By  Ostrowski's theorem $\vert\cdot\vert_\omega=\vert\cdot\vert^{\varepsilon(\omega)}$, therefore
$$
\vert x\vert^{\frac{1}{\varepsilon(\omega)}}_\omega\le\vert x-y\vert^{\frac{1}{\varepsilon(\omega)}}_{\omega}+\vert y\vert^{\frac{1}{\varepsilon(\omega)}}_\omega\le 1+\vert y\vert^{\frac{1}{\varepsilon(\omega)}}_\omega\,.
$$
This clearly means $\vert x\vert^{\frac{1}{\varepsilon(\omega)}}_\omega\le 2\max\left\{1, \vert y\vert^{\frac{1}{\varepsilon(\omega)}}_\omega\right\}$. After raising each side to the power of $\varepsilon(\omega)$ and applying $\log^+$ we get the claim.
\endproof

Now we prove the analogue of \cite[Proposition 3.1]{C}.

\begin{prop}\label{corvajaprop3.1} Let $(\mathbb K,\Omega, \phi)$ be a proper adelic curve. Fix a matrix $T(\bm{\beta})$. Let $\eta,s,t_1,\ldots,t_n\in\mathbb R$ such that all the hypotheses of Proposition \ref{interpolationpoly} are satisfied. Assume that $T(\bm{\beta})$ satisfies the $\lambda$-gap property of width $\eta$  (see Definition \ref{lambda_gap}) and that $\theta\colon S\to \mathbb R_{\ge 0}$ is $\lambda$-bounding for $T(\bm{\beta})$ (see Definition \ref{lambda_bou}), then
$$
\sum_{h=1}^{n}\begin{pmatrix}t_h-s-\frac{2\eta N^2}{V(s)}\end{pmatrix}\int_{ S_h}\theta d \mu\le N
$$
\end{prop}
\begin{proof} Fix a real number $D>0$, and for every $j=1,\dots, N$, we  put 
$$
d_{j}:=\frac{D}{\lambda(\bm{\alpha}^{(j)},\beta^{(j)})}
$$ 
and $d:=\prod^N_{j=1} d_j$. Note that here we use the $\lambda$-gap property to ensure that $\frac{d_{}j+1}{d_j}<\gamma$ since $\eta=2n\gamma$ as explained before Proposition \ref{interpolationpoly}. So, the hypotheses of Proposition \ref{interpolationpoly} are all satisfied and we have the interpolation polynomial $\delta$ such that $\delta(\bm{\beta})\neq 0$. We start by distinguishing two cases.

{\emph{First case: $\omega\in S_h$}}.  By the Taylor expansion for $\delta$ at $\bm{\alpha}_h$ we write
\begin{equation}\label{bound_pol}
\delta(\bm{\beta})=\sum_{{\bm i}}\Delta^{{\bm i}}\delta(\bm{\alpha}_h)\prod_{j=1}^{N}(\beta^{(j)}- \alpha_{h}^{(j)})^{i_j}\quad \textrm{where }{\bm{i}}=(i_1,\ldots, i_N)
\end{equation}
Notice that each term 
$$
 \Delta^{{\bm i}}\delta(\bm{\alpha}_h)\prod_{j=1}^{N}(\beta^{(j)}- \alpha_{h}^{(j)})^{i_j}
$$
is bounded from above by:
$$
\left(\max_{\bm{i}}\Delta^{{\bm i}}\delta(\bm{\alpha}_h)\right)\;\left(\max^\ast_{\bm{i}}\prod_{j=1}^{N}(\beta^{(j)}- \alpha_{h}^{(j)})^{i_j}\right)\,,
$$
where  in order to simplify the notation we define
$$
\max^\ast_{\bm{i}}:=\max_{\substack {\bm{i}\\ \Delta^{{\bm i}}\delta(\bm{\alpha}_h)\neq 0}}
$$
Let's put $M:=\prod^N_{j=1} (1+\deg_{X_j}\delta)$; so by taking the absolute value $\vert\cdot\vert_{\omega}$ in equation (\ref{bound_pol}) we get
\begin{equation}\label{bound_pol1}
\log|\delta(\bm{\beta})|_{\omega}\le \log^+\vert M\vert_\omega+\max_{\bm i}\,\log \left\vert\Delta^{{\bm i}}\delta(\bm{\alpha}_h)\right\vert_\omega+\max^\ast_{\bm{i}}\,\log
\left\vert\prod_{j=1}^{N}(\beta^{(j)}- \alpha_{h}^{(j)})^{i_j}\right\vert_{\omega}
\end{equation}
 The last summand of equation (\ref{bound_pol1}) is bounded from above by
$$
-{\Ind}_{\bm{\alpha}_h, \bm{d}}(\delta)\,\min\left\{d_1\log\frac{1}{\vert\beta^{(1)}- \alpha_{h}^{(1)}\vert_\omega},\ldots , 
d_N\log\frac{1}{\vert\beta^{(N)}- \alpha_{h}^{(N)}\vert_\omega}\right\}\,.$$
We  use Lemma \ref{corvajalemmepag166} to give an upper bound for $\log \left\vert\Delta^{{\bm i}}\delta(\bm{\alpha}_h)\right\vert_\omega$, hence from equation (\ref{bound_pol1}) we  deduce:
$$
\int_{ S_h}\log\vert \delta(\bm{\beta})\vert_\omega d \mu(\omega)\leq
$$
$$
\leq \int_{S_h}\log^+\lvert M\vert_\omega d\mu(\omega)+ \int_{ S_h}\left(h_\omega(\delta)+\log^+ \lvert M\rvert_\omega+\sum_{j=1}^{N}(\log^+\vert 2\vert_\omega+\log^+\vert\alpha_{h}^{(j)} \vert_\omega)\deg_{X_j}\delta\right)d\mu(\omega) +
$$
$$
-\int_{ S_h}\Ind_{\bm{\alpha}_h, \bm{d}} (\delta)\,\min\left\{d_1\log\frac{1}{\vert\beta^{(1)}- \alpha_{h}^{(1)}\vert_\omega},\ldots , 
d_N\log\frac{1}{\vert\beta^{(N)}- \alpha_{h}^{(N)}\vert_\omega}\right\}d\mu(\omega)
$$

So, by rearranging the terms  and summing over all $h=1,\ldots,n$ we get
\begin{multline}\label{long_ext1}
\sum^n_{h=1}\Ind_{\bm{\alpha}_h, \bm{d}} (\delta)\int_{ S_h}\min\left\{d_1\log\frac{1}{\vert\beta^{(1)}- \alpha_{h}^{(1)}\vert_\omega},\ldots , 
d_N\log\frac{1}{\vert\beta^{(N)}- \alpha_{h}^{(N)}\vert_\omega}\right\}d\mu(\omega)\leq -\int_{ S}\log\vert \delta(\bm{\beta})\vert_\omega d \mu(\omega)\\
+ 2\int_S \log^+\vert M\vert_\omega d\mu(\omega)+\int_{S}h_\omega(\delta) d\mu(\omega) +\sum_{h=1}^n\int_{S_h}\sum_{j=1}^{N}(\log^+\vert 2\vert_\omega+\log^+\vert\alpha_{h}^{(j)} \vert_\omega)\deg_{X_j}\delta d\mu(\omega)
\end{multline}
At this point  we can use Lemma \ref{log2} and Remark \ref{rem_forthelemma} for the following bound:
\begin{equation}\label{lastimachiave}
\sum_{h=1}^n\int_{S_h}\sum_{j=1}^{N}(\log^+\vert 2\vert_\omega+\log^+\vert\alpha_{h}^{(j)} \vert_\omega)\deg_{X_j}\!\!\delta d\mu(\omega)\le \int_{S} \sum_{j=1}^{N}(\log^+\vert 4\vert_\omega+\log^+\vert\beta_{h}^{(j)} \vert_\omega)\deg_{X_j}\!\!\delta d\mu(\omega)\,
\end{equation}
 to obtain
\begin{multline}\label{long_ext1}
\sum^n_{h=1}\Ind_{\bm{\alpha}_h, \bm{d}} (\delta)\int_{ S_h}\min\left\{d_1\log\frac{1}{\vert\beta^{(1)}- \alpha_{h}^{(1)}\vert_\omega},\ldots , 
d_N\log\frac{1}{\vert\beta^{(N)}- \alpha_{h}^{(N)}\vert_\omega}\right\}d\mu(\omega)\leq -\int_{S}\log\vert \delta(\bm\beta)\vert_\omega d \mu(\omega)+\\
+  2\int_S \log^+\vert M\vert_\omega d\mu(\omega)+\int_{ S}h_\omega(\delta) d\mu(\omega)+\int_{ S} \sum_{j=1}^{N}(\log^+\vert 4\vert_\omega+\log^+\vert\beta_{h}^{(j)} \vert_\omega)\deg_{X_j}\!\!\delta d\mu(\omega)\,.
\end{multline}

{\it{ Second case: $\omega\not\in  S$.}} Consider the expression:
$$
\delta(\bm{\beta})=\sum_{\bm i}\Delta^{\bm i}\delta(\bm{0})\beta^{(1)_{i_1}}\ldots \beta^{(N)_{i_N}} \quad \textrm{where }{\bm {i}}=(i_1,\ldots, i_N)
$$
We take the absolute value:
$$ \log\vert \delta(\bm{\beta})\vert_\omega\leq h_\omega(\delta)+\sum_{j=1}^{N}(\log^+\vert\beta^{(j)}\vert_\omega\deg_{X_j}\!\!\delta)+ \log^+\vert M\vert_\omega
$$
Hence
\begin{equation}\label{intermedioarchimedeo}
\int_{\Omega\setminus S}\log\vert \delta(\bm{\beta})\vert_\omega d\mu(\omega)\leq \int_{\Omega\setminus  S}h_\omega(\delta)d\mu(\omega)+\sum_{j=1}^{N}\int_{\Omega\setminus   S}(\log^+\vert\beta^{(j)}\vert_\omega\deg_{X_j}\!\!\delta\,)d\mu(\omega)+\int_{\Omega\setminus S}\log^+\vert M\vert_\omega d \mu(\omega)
\end{equation}
Since $(\mathbb K,\Omega, \phi)$ is proper
\begin{equation}\label{ad_prod}
-\int_{ S} \log \vert\delta(\bm{\beta})\vert_\omega d\mu(\omega)=\int_{\Omega\setminus  S} \log \vert\delta(\bm{\beta})\vert_\omega d\mu(\omega)\,.
\end{equation}
The distinction of the two cases is now finished, so by using equation (\ref{ad_prod}) and estimate  (\ref{intermedioarchimedeo}) inside (\ref{long_ext1}) we get 
\begin{multline}\label{verylongext}
\sum^n_{h=1}\Ind_{\bm{\alpha}_h, \bm{d}} (\delta)\int_{ S_h}\min\left\{d_1\log\frac{1}{\vert\beta^{(1)}- \alpha_{h}^{(1)}\vert_\omega},\ldots , 
d_N\log\frac{1}{\vert\beta^{(N)}- \alpha_{h}^{(N)}\vert_\omega}\right\}d\mu(\omega)\leq 2h(M) +\\
+\int_{\Omega}h_\omega(\delta)d\mu(\omega)+\sum_{j=1}^{N}\int_{\Omega\setminus  S}(\log^+\vert\beta^{(j)}\vert_\omega\deg_{X_j}\!\!\delta\,)d\mu(\omega)+\int_{ S} \sum_{j=1}^{N}(\log^+\vert 4\vert_\omega+\log^+\vert\beta_{h}^{(j)} \vert_\omega)\deg_{X_j}\!\!\delta d\mu(\omega)
\end{multline}
 Since we can always assume that $\mu$ is adequately normalised, we have $\int_{\Omega}\log^{+}\vert 4\vert_\omega d\mu(\omega)=h(4)\le\log 4$, therefore:
\begin{multline*}
\sum_{j=1}^{N}\int_{\Omega\setminus  S}(\log^+\vert\beta^{(j)}\vert_\omega\deg_{X_j}\!\!\delta\,)d\mu(\omega)+\int_{ S} \sum_{j=1}^{N}(\log^+\vert 4\vert_\omega+\log^+\vert\beta_{h}^{(j)} \vert_\omega)\deg_{X_j}\!\!\delta d\mu(\omega)\leq\\
\leq \sum_{j=1}^{N}\left(h(\beta^{(j)})+\log 4\right)\deg_{X_j}\!\!\delta
\end{multline*}
By plugging everything inside equation (\ref{verylongext}) we get
\begin{multline}\label{Vlong}
\sum^n_{h=1}\Ind_{\bm{\alpha}_h, \bm{d}} (\delta)\int_{ S_h}\min\left\{d_1\log\frac{1}{\vert\beta^{(1)}- \alpha_{h}^{(1)}\vert_\omega},\ldots, d_N\log\frac{1}{\vert\beta^{(N)}- \alpha_{h}^{(N)}\vert_\omega}\right\}d\mu(\omega)\leq \\
\leq 2h(M)+h(\delta)+\sum_{j=1}^{N}\left(h(\beta^{(j)})+\log 4\right)\deg_{X_j}\!\!\delta
\end{multline}
By the $\lambda$-bounding hypothesis for $T(\bm{\beta})$ 
we can write 
$$
-\frac{D}{{\lambda(\bm{\alpha}^{(j)},\beta^{(j)})}}\log\vert\beta^{(j)}- \alpha_{h}^{(j)}\vert_\omega \ge D\theta(\omega)\,. 
$$
Then we use Proposition \ref{interpolationpoly}$(3)$, so we can conclude 
\begin{multline}\label{sinistro}
D\sum_{h=1}^{n}dV(s)\left(t_h-s-\frac{2N^2\eta}{V(s)} \right)\int_{S_h}\theta d\mu\le\\
\leq\sum^n_{h=1}\Ind_{\bm{\alpha}_h, \bm{d}} (\delta)\int_{S_h}\min\left\{d_1\log\frac{1}{\vert\beta^{(1)}- \alpha_{h}^{(1)}\vert_\omega},\ldots, d_N\log\frac{1}{\vert\beta^{(N)}- \alpha_{h}^{(N)}\vert_\omega}\right\}d\mu(\omega)
\end{multline}
Now we are going to use again Proposition \ref{interpolationpoly}  to find upper bounds for the terms on the right hand side of equation (\ref{Vlong}). By Proposition \ref{interpolationpoly}$(2)$:
 \begin{equation}\label{destro}
 \sum_{j=1}^{N}\left(h(\beta^{(j)})+\log 4\right)\deg_{X_j}\!\!\delta\leq \sum_{j=1}^{N}\left( h(\beta^{(j)})+\log 4\right)(dd_jV(s)+O(d))\,.
 \end{equation}
Proposition \ref{interpolationpoly}$(4)$ says that,
\begin{equation}
h(\delta)\leq d\sum_{j=1}^{N}d_j\left(\log 2+ \sum_{h=1}^{n}V(t_h)h(\alpha^{(j)}_{h})\right)+O(d\log d)
\end{equation}
So now (\ref{Vlong}) can be written in the following way:
 \begin{multline}\label{thefinal}
 D\sum_{h=1}^{n}dV(s)\left(t_h-s-\frac{2N^2\eta}{V(s)} \right)\int_{ S_h}\theta d\mu\leq 2h(M)+d\sum_{j=1}^{N}d_j\left(\log 2+ \sum_{h=1}^{n}V(t_h)h(\alpha^{(j)}_{h})\right)+ \\
 +O(d\log d)+\sum_{j=1}^{N}\left( h(\beta^{(j)})+\log 4\right)(dd_jV(s)+O(d))
\end{multline}
Notice that since $s\le 1$ and $N\ge 2$ then $V(s)\le\frac{1}{2}$, which is equivalent to the inequality $V(s)\log 4+\log 2\leq \log 4$. Then the expression
$$
d\sum_{j=1}^{N}d_j\left(\log 2+ \sum_{h=1}^{n}V(t_h)h(\alpha^{(j)}_{h})\right)+O(d{\rm{log}}d)+\sum_{j=1}^{N}\left( h(\beta^{(j)})+\log 4\right)(dd_jV(s)+O(d))
$$
is bounded by 
$$
d\sum_{j=1}^{N}d_j\left(V(s) h(\beta^{(j)})+\log 4+   \sum_{h=1}^{n}V(t_h)h(\alpha^{(j)}_{h})\right)+O(d\log d)
$$
But by the definitions of $d_j$ and the function $\lambda$ it holds that
\begin{equation}
dV(s)\sum_{j=1}^{N}d_j\left(h(\beta^{(j)})+\frac{\log 4}{V(s)}+   \sum_{h=1}^{n}\frac{V(t_h)}{V(s)}h(\alpha^{(j)}_{h})\right)=dV(s)ND\,.
\end{equation}
Therefore equation (\ref{thefinal}) becomes
$$
 dDV(s)\sum_{h=1}^{n}\left(t_h-s-\frac{2N^2\eta}{V(s)} \right)\int_{ S_h}\theta d\mu\leq dV(s)ND+O(d\log d)
$$
since the terms $2h(M)$ and $O(d)$ can be put together inside $O(d\log d)$. Finally we divide both sides by $dDV(s)$ and we take the limit for $D\to+\infty$ (i.e. $d\to+\infty$) to conclude the proof.
\end{proof}

\begin{thm}\label{teoremadue}  Let $(\mathbb K,\Omega, \phi)$ be a proper adelic curve. Assume that the matrix $T(\bm{\beta})\in M(n\times N,\mathbb K)$ satisfies the h-gap property (see Definition \ref{boxing}). Moreover fix $N>21^2\log 2n$ and let $\theta\colon S\to\mathbb R_{\ge 0}$ be a column-bounding function for $T(\bm{\beta})$ (see Definition \ref{bounding}). Then
$$
\int_{ S}\theta d \mu<2+\frac{7\sqrt{\log 2n}}{\sqrt{N}}
$$
\end{thm}

\proof The proof is based on the exact calculation of the volume $V(s)$ of \cite[Chap. 2, Lemma 5A, 5B]{schm}. It shows that for  $N\geq 2$, $s,\rho\in\mathbb R_{>0}$ with $0<s\leq 1$ and $0<\rho <\frac{N}{2}$, we have:
\begin{equation}\label{stima_sch}
V(s)= \frac{s^N}{N!}
\end{equation}
\begin{equation}\label{stima_sch2}
V\left(\frac{N}{2}-\rho\right)<e^{\frac{-\rho^2}{N} }
\end{equation}
We take $t_1=t_2=\ldots= t_n=\frac{N}{2}-\rho$, $\eta=\frac{1}{2N^2 N!}$ where $\rho$ is a number such that:
$$
V\left(\frac{N}{2}-\rho\right)=\frac{1-\eta N^2}{n}
$$
Note that $\frac{1-\eta N^2}{n}>\frac{1}{2n}$. Now from equation (\ref{stima_sch2}) we get
$$
\frac{-\rho^2}{N}>\log\left(\frac{1-\eta N^2}{n}\right)
$$
which implies easily
$$
\rho< \sqrt {N\log2n}\,.
$$
Now let's take $s$ as a solution of:
\begin{equation}\label{stimavol}
(1+\eta)^{N}-1+\eta N^2<V(s)< 2N\eta +\eta N^2
\end{equation}
Here recall that $1-\eta N^2=nV(t_h)=\sum^n_{h=1} V(t_h)$ (compare with the condition (\ref{basicinequality})).
From equation (\ref{stimavol}) we get the following two conditions
\[
N!<\frac{1}{V(s)}<2N!
\]
\[
\frac{N!}{2n}<\frac{V(t_h)}{V(s)}<\frac{2N!}{n}
\]

Now we show that by using such conditions since $T(\bm{\beta})$ satisfies $h$-gap property, then it satisfies also the $\lambda$-gap property of width $\eta$. First let's use the inequalities $\frac{1}{V(s)}<2N!$  and $\frac{V(t_h)}{V(s)}<\frac{2N!}{n}$ to bound $\lambda(\bm{\alpha}^{(j)},\beta^{(j)})$ from above:
\[
\lambda(\bm{\alpha}^{(j)},\beta^{(j)})=\frac{1}{V(s)}{\rm{log}}4+h(\beta^{(j)})+\frac{1}{V(s)}\sum_{h=1}^{n}V(t_h)h(\alpha_{h}^{(j)})<
\]
\[
< 2N!\log4+h(\beta^{(j)})+\frac{2N!}{n}\sum_{h=1}^{n}h(\alpha_{h}^{(j)})=\log \rho_j\,.
\]
Then we use the inequalities $N!<\frac{1}{V(s)}$ and $\frac{N!}{2n}<\frac{V(t_h)}{V(s)}$ to bound $\lambda(\bm{\alpha}^{(j+1)},\beta^{(j+1)})$ from below:
\[
\lambda(\bm{\alpha}^{(j+1)},\beta^{(j+1)})=\frac{1}{V(s)}{\rm{log}}4+h(\beta^{(j+1)})+\frac{1}{V(s)}\sum_{h=1}^{n}V(t_h)h(\alpha_{h}^{(j+1)})>
\]
\[
> N!\log4+h(\beta^{(j+1)})+\frac{N!}{2n}\sum_{h=1}^{n}h(\alpha_{h}^{(j+1)})=\log \rho'_j\,.
\]
Therefore now:
\[
\frac{\lambda(\bm{\alpha}^{(j)},\beta^{(j)})}{\lambda(\bm{\alpha}^{(j+1)},\beta^{(j+1)})}<\frac{\log \rho_j}{\log \rho'_j}<\frac{1}{4nN^2N!}=\frac{\eta}{2n}\,.
\]

Note that as an immediate consequence of the last inequality  it follows that all the hypotheses of Proposition \ref{interpolationpoly} are satisfied.  Moreover, since $\theta$ is a column-bounding function for $T(\bm{\beta})$, then it is $\lambda$-bounding for $T(\bm{\beta})$. It means that we can apply Proposition \ref{corvajaprop3.1} to get:
\begin{equation}\label{lastupb}
\int_{ S}\theta d \mu\le \frac{N}{\frac{N}{2}-\rho-s-\frac{2\eta N^2}{V(s)} }
\end{equation} 
Since  $\rho< \sqrt{N\log 2n}$,  $N^2\eta<V(s)$ and in particular  $N>\max\left\{\frac{36}{\log 2n}, 9\log 2n\right\}$, we obtain:

\[
\frac{2(\rho+s)}{N}+\frac{4\eta N}{V(s)}<\frac{2\sqrt{N\log2n}+2+4}{N}<\frac{3\sqrt{N\log 2n}}{N}=\frac{3\sqrt{\log 2n}}{\sqrt N}<1
\]
Therefore
\[
\frac{N}{\frac{N}{2}-\rho-s-\frac{2\eta N^2}{V(s)}}=\frac{2}{1-\frac{2(\rho+s)}{N}-  \frac{4\eta N}{V(s)}}<\frac{2}{1-\frac{3\sqrt{\log 2n}}{\sqrt N}}<2+\frac{7\sqrt{\log 2n}}{\sqrt{N}}\,.
\]
Where the last inequality follows by the assumption $N>21^2\log 2n$.
\endproof

\section{Roth's theorem for adelic curves (A)}\label{mainsect}

\begin{defn}\label{st_muequi}
A proper adelic curve $\mathbb X=(\mathbb K,\Omega, \mu)$ satisfies the \emph{ strong $\mu$-equicontinity condition} if for any measurable set $S\subset\Omega$ of finite measure and any real number $\varepsilon>0$ there exists a finite measurable cover $C_1,\ldots, C_m$ of $S$ satisfying the following conditions: for all $\beta\in\mathbb K^\times$ there exists a measurable set $U_{\beta}\subset\Omega_\infty$ such that $\mu(U_\beta)=0$ and
\[
\left\vert-\log^-\vert\beta\vert_{\omega}+\log^-\vert\beta\vert_{\omega'} \right\vert<\varepsilon h(\beta)\,,\quad \forall \omega,\omega'\in C_j\setminus U_\beta\,,\; \forall j=1,\ldots, m
\]
\end{defn}

The following lemma is a simplified version of \cite[Lemma 8.10]{vo}. It can be seen as a generalisation of the Arzel\`a-Ascoli theorem for measure spaces, with the advantage that one doesn't need to provide a uniform bound for the involved family of functions.

\begin{lem}\label{unif_conv}
Let $\mathbb X=(\mathbb K,\Omega, \phi)$ be a proper adelic curve satisfying the strong $\mu$-equicontinuity condition. Fix $\alpha\in \mathbb K^\times$ and a measurable subset $S\subset\Omega$ of finite measure. Let $\{\beta_k\}$ be a sequence in $\mathbb K^\times$ such that $\beta_k\neq\alpha$,  and $h(\beta_k)\to+\infty$. Then for any $\varepsilon>0$  there exists a subsequence $\left\{\beta_{k_j}\right\}$  such that for any $j,\ell\in\mathbb N$  big enough  the following inequality holds  on $S\setminus U$, where $U\subset\Omega_{\infty}$ and $\mu(U)=0$:
\[
\left\vert-\frac{\log^-\vert\alpha-\beta_{k_j}\vert_{\omega}}{h(\beta_{k_j})}+\frac{\log^-\vert\alpha-\beta_{k_{\ell}}\vert_{\omega}}{h(\beta_{k_{\ell}})} \right\vert<\varepsilon 
\]
\end{lem}
\proof 
Put $h_0=\min_k h(\beta_k)$ and $\lambda_k(\omega)=-\log^-\vert \alpha-\beta_k\vert_\omega$. Fix $\varepsilon>0$ and let $C_1,\ldots, C_r$ be a finite measurable cover as in Definition \ref{st_muequi}. Clearly by possibly passing to a refinement we can assume that the cover is made of mutually disjoints sets. Let $k\in\mathbb N$; if $i\in\{1,\ldots, r\}$ is such that $\mu(C_i)=0$ we define $m_{k,i}=0$, otherwise if $i$ is such that $\mu(C_i)>0$ we put 
\[
m_{k,i}=\inf\left\{t\in\mathbb R\colon \mu\left(\left\{\omega\in C_i\colon \frac{\lambda_k(\omega)}{h(\beta_k)}\ge t\right\}\right)\le \frac{\mu(C_i)}{2}\right\}
\]
Notice that the sets:
\[
V_{k,i}=\{\omega\in C_i\colon \lambda_k(\omega) \le h(\beta_k) m_{k,i}\}\,,\quad  T_{k,i}=\{\omega\in C_i\colon \lambda_k(\omega) \ge h(\beta_k) m_{k,i}\}
\]
have both measure at least $\frac{\mu(C_i)}{2}$. Hence we get:
\begin{equation}\label{ee1}
\frac{\mu(C_i)}{2} h(\beta_k) m_{k,i}\le \int_{T_{k,i}} h(\beta_k) m_{k,i} d\mu \le \int_{T_{k,i}} \lambda_k d\mu\le\int_S \lambda_k d\mu
\end{equation}
But by Proposition \ref{properties}(6) we know that
\begin{equation}\label{ee2}
\int_S \lambda_k d\mu\le h(\beta_k)\left(\frac{\log 2}{h(\beta_k)}+\frac{h(\alpha)}{h(\beta_k)}+1\right)\le h(\beta_k)\left(\frac{\log 2}{h_0}+\frac{h(\alpha)}{h_0}+1\right)
\end{equation}
So by putting $c_i=\frac{2}{\mu(C_i)}\left(\frac{\log 2}{h_0}+\frac{h(\alpha)}{h_0}+1\right)$, equations  (\ref{ee1}) and (\ref{ee2}) show that $m_{k,i}\le c_i$. Note that the constant $c_i$ doesn't depend on $k$. It follows that all the vectors $\bm{m}_{k}=(m_{k,1},\ldots, m_{k,r})\in\mathbb R^r$,   lie in the hyper-parallelepiped $\prod^r_{i=1}[0,c_i]$. Now consider the small hyper-cubes $Q_k=\prod^r_{i=1}[m_{k,i}-\varepsilon,m_{k,i}+\varepsilon]$, clearly there exists an index $\overline{k}\in\mathbb N$ and a sequence  $\{k_j\}$ such that $\bm{m}_{k_j}\in Q_{\overline k}$ for any $j\in\mathbb N$. We now show that the subsequence $\{\beta_{k_j}\}$ has exactly the properties that we are searching for.

 Consider any two elements $\beta_{k_j}$ $\beta_{k_{\ell}}$ of the subsequence and  $\omega\in S\setminus( U_{\beta_{k_j}}\cup U_{\beta_{k_\ell}})$. Let $i$ be the unique index such that $\omega\in C_i$; pick $\omega'\in C_i
\setminus U_{\beta_{k_j}}$  and  $\omega''\in C_i
\setminus U_{\beta_{k_\ell}}$ such that:
\[
\lambda_{k_j}(\omega')\le m_{k_j,i} h(\beta_{k_j})\,,
\]
\[
\lambda_{k_\ell}(\omega'')\ge m_{k_\ell,i} h(\beta_{k_\ell})\,.
\]
Note that this is possible since $V_{k_j,i}$ and $T_{k_\ell,i}$ have positive measure.  In the following chain of inequalities we use twice  the strong $\mu$-equicontinuity condition (first for $\omega,\omega'$ and later for $\omega,\omega''$) and the fact that $|m_{k_j,i}-m_{k_\ell,i}|<2\varepsilon$.

\[
\frac{\lambda_{k_j}(\omega)}{h(\beta_{k_j})}< \frac{\lambda_{k_j}(\omega')}{h(\beta_{k_j})}+ \frac{h(\alpha-\beta_{k_j})}{h(\beta_{k_j})}\varepsilon<\!^{(\text{Pr. \!\!\ref{properties}(4)})}\frac{\lambda_{k_j}(\omega')}{h(\beta_{k_j})}+\left(1+\frac{\log 2+h(\alpha)}{h(\beta_{k_j})}\right)\varepsilon \le 
\]
\[
\le m_{k_j,i}+\left(1+\frac{\log 2+h(\alpha)}{h(\beta_{k_j})}\right)\varepsilon\le m_{k_\ell,i}+\left(3+\frac{\log 2+h(\alpha)}{h(\beta_{k_j})}\right)\varepsilon\le \frac{\lambda_{k_\ell}(\omega'')}{h(\beta_{k_\ell})}+\left(3+\frac{\log 2+h(\alpha)}{h(\beta_{k_j})}\right)\varepsilon< 
\]
\[
<\frac{\lambda_{k_\ell}(\omega)}{h(\beta_{k_\ell})}+\left(4+\frac{\log 2+h(\alpha)}{h(\beta_{k_j})}+\frac{\log 2+h(\alpha)}{h(\beta_{k_\ell})}\right)\varepsilon
\]
Since $h(\beta_k)\to +\infty$  for $j$ and $\ell$ big enough the above inequalities say:
\[
\frac{\lambda_{k_j}(\omega)}{h(\beta_{k_j})}< \frac{\lambda_{k_\ell}(\omega)}{h(\beta_{k_\ell})}+5\varepsilon
\]
By swapping the roles of $j$ and $\ell$ we obtain in the same way
\[
\frac{\lambda_{k_\ell}(\omega)}{h(\beta_{k_\ell})}< \frac{\lambda_{k_j}(\omega)}{h(\beta_{k_j})}+5\varepsilon
\]
 The claim finally follows by the arbitrariness of $\varepsilon$.
\endproof

\textbf{Proof of theorem \bref{thm:main}{(A)}.}
When  $n \le 2$, from  Proposition \ref{properties}$(6)$ we deduce that  any approximant $\beta$ such that 
$$
h(\beta)>\frac{n\log 2+\sum^n_{i=1}h(\alpha_i)}{2-n+\varepsilon}
$$
satisfies the desired inequality.\\

 Assume $n\ge 3$. We consider the following matrix of dimension $n\times N$, where $N>21^2\log 2n$ will be a big enough fixed integer:
\[
T=\begin{pmatrix}
 \alpha_1& \alpha_1  &\ldots &\alpha_1 \\
 \alpha_2& \alpha_2  &\ldots &\alpha_2\\
\vdots & \vdots  & \vdots & \vdots\\
\alpha_n& \alpha_n  &\ldots &\alpha_n\\
\end{pmatrix}
\]
notice that we have repeated $N$ times the same column vector. Assume by contradiction that the theorem is false. Namely that there exists an $\varepsilon_0>0$ such that the inequality
\begin{equation}\label{negationroth}
\sum^n_{i=1}\int_{S_i}\log^-\vert \beta_k-\alpha_i\vert_\omega d\mu(\omega)\le-(2+\varepsilon_0)h_\mathbb X(\beta_k)
\end{equation}
is satisfied by  a sequence $\{\beta_k\}_{k\in\mathbb N}$ in $\mathbb K$ with the properties that $h(\beta_k)\to+\infty$. Pick a constant:
\[
L>  \log\left(4^{2N!}\prod^n_{i=1}H(\alpha_i)^{\frac{2N!}{n}}\right) 
\]
Since $h(\beta_k)\to+\infty$, by eventually passing to a subsequence,  we can assume that  $\{h(\beta_k)\}$ is increasing and bounded from below by a very  big value. Therefore we can assume
\begin{equation}\label{magg-eps}
\varepsilon_0>\frac{2L}{h(\beta_k)}+\frac{L+h(\beta_k)}{h(\beta_k)}\frac{8\sqrt{\log 2n}}{\sqrt N}\,,\quad \forall  k\in\mathbb N
\end{equation}
Choose  any $N$ elements from the sequence $\{\beta_{k}\}$, call them $\beta^{(1)},\ldots\beta^{(N)}$ and consider the following matrix $T(\bm\beta)$:
\[
T(\bm\beta)=\begin{pmatrix}
 \alpha_1& \alpha_1  &\ldots &\alpha_1 \\
 \alpha_2& \alpha_2  &\ldots &\alpha_2\\
\vdots & \vdots & \vdots & \vdots\\
\alpha_n& \alpha_n  &\ldots &\alpha_n\\
\beta^{(1)}& \beta^{(2)}  &\ldots &\beta^{(N)}
\end{pmatrix}
\]

We are ready to construct some functions  $\theta_k\in L^1(S,\mu)$ which will give the desired contradiction. We define them as  piecewise functions by putting for any $\omega\in  S_i$ and any $k\in\mathbb N$:
\[
\theta_k(\omega)=\frac{-\log^- \vert \beta_k-\alpha_i\vert_{\omega}}{L+h(\beta_k)}
\]
  Now we can write
\begin{equation}\label{magg-int}
\int_{S} \theta_k d\mu =\sum^n_{i=1}\int_{S_i} \theta_k d\mu=\frac{-1}{L+h(\beta_k)}\sum^n_{i=1}\int_{S_i}\log^- \vert \beta_k-\alpha_i\vert_{\omega} d\mu(\omega)\ge^{\!\!\textrm{eq.} (\ref{negationroth})}\frac{(2+\varepsilon_0)h(\beta_k)}{L+h(\beta_k)}
\end{equation}
By plugging inequality (\ref{magg-eps}) inside (\ref{magg-int}) and simplifying the expressions, we finally get:
\begin{equation}\label{magg-inte}
\int_{S}\theta_k d\mu >2+\frac{8\sqrt{\log 2n}}{\sqrt N} \qquad \forall k\in\mathbb N
\end{equation}
Thanks to Proposition \ref{properties}(6):
\[
\int_{S_i} \theta_k \le \frac{B_i+h(\beta_k)}{L+h(\beta_k)}
\]
for $B_i\in\mathbb R_{>0}$. So, if $B=\max_{i}\{B_i\}$ we obtain:
$$
\int_S\theta_k d\mu\le \frac{nB+nh(\beta_k)}{L+h(\beta_k)}\le n\,.
$$
Then after possibly passing to a subsequence of $\{\theta_k\}$, we can assume that $\int_{S}\theta_k d\mu$ admits limit and by inequality (\ref{magg-inte}):
\begin{equation}\label{key_in}
\lim_{k\to\infty }\int_{S}\theta_k d\mu\ge 2+\frac{8\sqrt{\log 2n}}{\sqrt N}>2+\frac{7\sqrt{\log 2n}}{\sqrt N}\,.
\end{equation}
Put by simplicity $\lambda_{i,k}(\omega):=-\log^- \vert \beta_k-\alpha_i\vert_{\omega}$ for $\omega\in S_i$, so 
$$\theta_k(\omega)=\frac{\lambda_{i,k}(\omega)}{L+h(\beta_k)}\,.$$
Then we have:
$$
\theta_k(\omega)=\frac{\lambda_{i,k}(\omega)}{h(\beta_k)}A_k
$$
where $A_k:=1-\frac{L}{L+h(\beta_k)}<1$ and $A_k\to 1$. By Lemma \ref{unif_conv} from $\left\{\frac{1}{A_k}\theta_k\right\}$ we can extract a subsequence  $\left\{\frac{1}{A_{k_j}}\theta_{k_j}\right\}$ uniformly convergent to a function $\theta$  on $S\setminus U$, where $\mu(U)=0$. We extend $\theta$ on the whole $S$ by putting $\theta_{|U}=0$. Since on a set of finite measure the uniform convergence implies the $L^1$-convergence, we can write:
\begin{equation}\label{final_contr}
    \begin{aligned}
\int_{S} \theta d\mu=\int_{S\setminus U} \theta d\mu=\lim_{j\to\infty}\int_{S\setminus U} \frac{1}{A_{k_j}}\theta_{k_j} d\mu=\lim_{j\to\infty}\int_{S} \frac{1}{A_{k_j}}\theta_{k_j} d\mu=\\
=\left(\lim_{j\to\infty}\frac{1}{A_{k_j}}\right)\left(\lim_{j\to\infty} \int_{S}\theta_{k_j} d\mu\right)> 2+\frac{7\sqrt{\log 2n}}{\sqrt N}\,.
\end{aligned}
\end{equation}
Consider a positive real number $C$ such that
\[
C\mu(S) < \int_S\theta d\mu - \left ( 2+ \frac{7\sqrt{\log 2n}}{\sqrt N} \right)
\]
and define the set
\[
T:=\{\omega\in S\colon \theta(\omega)< C\}\,.
\]
Note that $T$ contains the previously introduced $U$. We obtain the following chain of inequalities:
\[
\int_T \theta d\mu \le C\mu(T) \le C\mu(S) < \int_S\theta d\mu- \left ( 2+ \frac{7\sqrt{\log 2n}}{\sqrt N} \right)
\] 
which gives
\[
\int_{S\setminus T} \theta d\mu > 2+\frac{7\sqrt{\log 2n}}{\sqrt N}
\]
Now we define the function $\hat\theta: S\to \mathbb R_{\ge 0}$  such that
\begin{equation*}
\hat \theta(\omega)=\begin{cases}
\theta(\omega) & \text{if } \omega\in S\setminus T\\
0 & \text{otherwise}
\end{cases}
\end{equation*}
Clearly $\hat\theta(\omega)\ge C$ for any $\omega\in S\setminus T$ and moreover
\[
\int_{S} \hat\theta d\mu=\int_{S\setminus T} \theta d\mu > 2+\frac{7\sqrt{\log 2n}}{\sqrt N}\,.
\]
We can choose $\gamma\in\, ]0,1[$ such that
\begin{equation}\label{modification}
\int_{S}\gamma\hat\theta  d\mu > 2+\frac{7\sqrt{\log 2n}}{\sqrt N} \,.
\end{equation}
We want to show that  the function $\gamma\hat\theta$ satisfies  the hypotheses of Theorem \ref{teoremadue}. Choose $\delta\in\mathbb R_{>0}$ such that
\[
0<\delta<(1-\gamma)C
\]
so that for any $\omega\in S\setminus T $ we have
\[
\gamma\hat\theta(\omega) <\hat\theta(\omega)-\delta\,.
\]
Denote by $\{\theta_m\}$ the previously extracted subsequence that converges uniformly to $\theta$ on $S\setminus T$. We can find a positive integer $M$ such that if $m>M$ the following chain of inequalities holds for any $\omega\in S\setminus T$:
\begin{equation}\label{modification1}
0 < \gamma\hat\theta(\omega) < \hat\theta(\omega)-\delta<\theta_m(\omega)<\frac{-\log^{-}\vert\beta_{m}-\alpha_i\vert_\omega}{h(\beta_{m})+\log\left(4^{2N!}\prod^n_{i=1}H(\alpha_i)^{\frac{2N!}{n}}\right) }
\end{equation}
 We choose $\bm\beta:=(\beta^{(1)},\ldots,\beta^{(N)})$ in the set $\{\beta_m\}_{m>M}$ with the following properties:
\begin{itemize}
\item The elements $\beta^{(j)}$ are separated by big enough gaps, so that the matrix $T(\bm\beta)$ satisfies the $h$-gap condition.

\item  The elements $\beta^{(j)}$ are chosen in a way that $\vert \beta^{(j)}-\alpha_i\vert_\omega \le 1$ for any $i=1,\ldots, n$ and any $\omega\in S\setminus T$. Note that this is possible because we can iterate $n$-times Lemma \ref{unif_conv} starting with the sequence $\{\beta_m\}$ and extracting at each step $i=1,\ldots,n$ a sequence $\{\beta_j\}$ such that $\omega\mapsto\frac{-\log^-\vert\beta_j-\alpha_i\vert_\omega}{h(\beta_j)}$  converges uniformly almost everywhere to a function bounded from below by a positive number. In this way inequality (\ref{modification1}) ensures that $\gamma\hat\theta$ is column bounding.
\end{itemize}

Finally we can apply Theorem \ref{teoremadue} for $\gamma\hat\theta$  to get:
$$
\int_{ S}\gamma \hat\theta d \mu<2+\frac{7\sqrt{\log 2n}}{\sqrt{N}}
$$
which  contradicts inequality (\ref{modification}).\hfill{$\square$}

\vspace{1ex}
Now we show that Theorem \ref{CorCorv} can be immediately recovered from Theorem \bref{thm:main}{(A)}.

\begin{prop}\label{im_cor}
    Theorem \bref{thm:main}{(A)} implies Theorem \ref{CorCorv}.
\end{prop}
\proof
Let $\mathbb K$  be the Galois closure of $k(\alpha_1,\ldots,\alpha_n)$ over $k$ and consider the adelic curve  $\mathbb X= (\mathbb K,\Omega, \mu)$  naturally lying over $(k,\mathcal V_k)$. The strong $\mu$-equicontinuity condition trivially holds on $\mathbb X$ since all singletons of $\Omega$ are measurable and $\inf_{\omega\in\Omega}\mu(\{\omega\})>0$. For any  $v_i$ of Theorem \ref{CorCorv} consider the set
\[
\widetilde S_i:=\{\omega\in\Omega \colon \text{$\omega$ extends $v_i$}\}=\{\omega_{ij}\colon j=1,\ldots, r(i)\}\quad \text{for $i=1,\ldots,n$}
\]
The set $\widetilde S_i$ are orbits under the action of $\Gal(\mathbb K/k)$ on $\Omega$. For each $i,j$ there is $\alpha_{ij}\in \mathbb K$ such that
\[
\vert \beta-\alpha_{ij}\vert_{\omega_{ij}}=\vert \beta-\alpha_{i}\vert_{v_i},\quad \forall\beta\in k
\]
Now we apply Theorem \bref{thm:main}{(A)} on the set $\widetilde S=\cup_\ell\widetilde S_\ell$ partitioned along the fibers of the map $\omega_{ij}\mapsto a_{ij}$. Hence we get the following inequality
\[
\sum_{i,j}\int_{\{\omega_{ij}\}}\log^-\vert \beta-\alpha_{ij}\vert_{\omega_{ij}} d\mu(\omega)>-(2+\varepsilon) h_{\mathbb X}(\beta).
\]
Note that the summation over $i$ and $j$ doesn't reflect the partition. But 
\[
\sum_{i,j}\int_{\{\omega_{ij}\}}\log^-\vert \beta-\alpha_{ij}\vert_{\omega_{ij}} d\mu(\omega)=\sum_{ij} \log^-\vert \beta-\alpha_{ij}\vert_{\omega_{ij}}\mu(\{\omega_{ij}\})=
\]

\[
=\sum_i\left(\log^-\vert \beta-\alpha_{i}\vert_{v_i}\sum_j\mu(\{\omega_{ij}\})\right)=\sum_i\log^-\vert \beta-\alpha_{i}\vert_{v_i}
\]
Where the last equality follows from \cite[Ch. II, Corollary 8.4]{Neu}:
\[
\sum^{r(i)}_{j=1}\mu(\{\omega_{ij}\})=\sum^{r(i)}_{j=1}\frac{[\mathbb K_{\omega_{ij}}\colon k_{v_i}]}{[\mathbb K\colon k]}=1\,.
\]
\endproof

\section{Roth's theorem for adelic curves (B)}\label{mainsect2}

We weaken Definition \ref{st_muequi} by requiring the ``equicontinuity property'' outside from sets of arbitrary small measure:
\begin{defn}\label{muequi}
A proper adelic curve $\mathbb X=(\mathbb K,\Omega, \mu)$ satisfies the \emph{$\mu$-equicontinuity condition} if for any measurable set $S\subset\Omega$ of finite measure and any real numbers $\varepsilon,\delta>0$ there exists a finite measurable cover $C_1,\ldots, C_m$ of $S$ satisfying the following conditions. For all $\beta\in\mathbb K^\times$ there exists a measurable set $U_\beta\subset\Omega_\infty$ such that $\mu(U_\beta)<\delta$ and
\[
\left\vert-\log^-\vert\beta\vert_{\omega}+\log^-\vert\beta\vert_{\omega'} \right\vert<\varepsilon h(\beta)\,,\quad \forall \omega,\omega'\in C_j\setminus U_\beta\,,\; \forall j=1,\ldots, m
\]
\end{defn}

In \cite[Proposition 8.9]{vo} it is shown that  arithmetic function fields satisfy the $\mu$-equicontinuity condition.

\begin{defn}\label{un_int}
A proper adelic curve $\mathbb X=(\mathbb K,\Omega, \mu)$ satisfies the \emph{uniform integrability condition} if for  any $\varepsilon>0$ there exists $\delta>0$ such that if $T\subset\Omega_{\infty}$ is a measurable subset satisfying $\mu(T)<\delta$, then 
$$\int_T -\log^-|\beta|_\omega d\mu(\omega)<\varepsilon h_{\mathbb X}(\beta)\,,\quad \forall\beta\in\mathbb K^\times\,.$$
\end{defn}

Vojta shows that arithmetic function fields satisfy the uniform integrability condition in \cite[Lemma 8.8]{vo}.  The following lemma provides a key step in our proof:
\begin{lem}\label{vo8.12}
Let $\mathbb X=(\mathbb K,\Omega, \phi)$ be a proper adelic curve satisfying the $\mu$-equicontinuity condition  and the uniform integrability condition. Assume that Theorem \bref{thm:main}{(B)} doesn't hold for certain $S$, $\alpha_1,\ldots,\alpha_n\in\mathbb K$, $\varepsilon_0\in\mathbb R_{>0}$ and $c\in\mathbb R$. Let $N>0$ an integer, $\varepsilon\in\,]0,\varepsilon_0[$, and $r_0,r_1 >1$ two real numbers. Then there exist $\beta^{(1)},\ldots,\beta^{(N)}\in\mathbb K$ satisfying the following conditions:
\begin{itemize}
\item[$(1)$] $h(\beta^{(1)})>r_0$ 
\item[$(2)$] $\frac{h(\beta^{(k)})}{h(\beta^{(k-1)})}>r_1$ for any $k=2,\ldots, N$.
\item[$(3)$] There is a partition $S=S_1\sqcup\ldots\sqcup S_n$ such that:
$$\sum^n_{i=1}\int_{S_i}\, \min_{1\le k\le N}\left(\frac{-\log^-|\beta^{(k)}-\alpha_i|_\omega}{h(\beta^{(k)})}\right )d\mu(\omega)>2+\varepsilon$$
\end{itemize}
\end{lem}
\proof
The rather technical  proof can be found in \cite[Proposition 8.12]{vo} for arithmetic function fields, and it can be repeated line by line for our case. The proof uses \cite[Lemmas 8.10, 8.11]{vo}, that in turn rely on $3$ assumptions denoted in \cite{vo} by (i), (ii), (iii). Vojta needs to show that such assumptions are satisfied in his case, and the proofs are quite involved. On the other hand, we now  explain why the assumptions hold immediately in our setting: assumption (i) follows by the integral Liouville inequality (see Proposition \ref{properties}(6)) if we put $\Xi=\{\beta_k\}$, $q=n$, $\lambda_{\beta_k,j}=-\log^-\vert\alpha_j-\beta_k\vert$ and  $c_9= \log 2+\max_j h(\alpha_j)$ for $j=1,\ldots, n$. Assumption (ii) is the $\mu$-equicontinuity condition. Assumption (iii) is the uniform integrability condition.

Finally,  note that in the statement of \cite[Proposition 8.12]{vo} the condition $(1)$ is not explicitly mentioned as a consequence of the hypotheses, but it can be easily deduced from the proof.
\endproof

\textbf{Proof of theorem \bref{thm:main}{(B)}.}
When $n=1$, from  Proposition \ref{properties}$(6)$ we deduce that  any approximant $\beta$ such that 
$$h(\beta)>\frac{\log 2+h(\alpha_1)+c}{1+\varepsilon}
$$
satisfies the desired inequality.\\

 Fix $n\ge 2$. Assume by contradiction that the theorem is false for some  $\varepsilon_0>0$ and that the counterexample is given by a sequence of approximants $\{\beta_k\}$ such that $h(\beta_k)\to+\infty$.  We apply Lemma \ref{vo8.12}; so for any  $r_0,r_1>1$ and $\varepsilon<\varepsilon_0$  there exist $\beta^{(1)},\ldots,\beta^{(N)}\in\{\beta_m\}$ satisfying the properties $(1)$-$(3)$. Consider the matrix constructed with the vector $\bm{\beta}=(\beta^{(1)},\ldots,\beta^{(N)})$:
\[
T(\bm{\beta})=\begin{pmatrix}
 \alpha_1& \alpha_1  &\ldots &\alpha_1 \\
 \alpha_2& \alpha_2  &\ldots &\alpha_2\\
\vdots & \vdots  & \vdots & \vdots\\
\alpha_n& \alpha_n  &\ldots &\alpha_n\\
\beta^{(1)}& \beta^{(2)}  &\ldots &\beta^{(N)}
\end{pmatrix}
\] 
 Pick a constant:
\[
L>  \log\left(4^{2N!}\prod^n_{i=1}H(\alpha_i)^{\frac{2N!}{n}}\right) 
\]
 We are ready to construct a function  $\theta\in L^1(S,\mu)$ which will give the desired contradiction. We define it as a piecewise function by putting for any $\omega\in S_i$:
\[
\theta|_{S_i}(\omega)=\min_k\frac{-\log^- \vert \beta^{(k)}-\alpha_i\vert_{\omega}}{L+h(\beta^{(k)})}= \min_k \frac{-\log^- \vert \beta^{(k)}-\alpha_i\vert_{\omega}}{h(\beta^{(k)})}A_k
\]
where $A_k=1-\frac{L}{L+h(\beta^{(k)})}$. Let's fix the set:
\[
\hat S_i:=\{\omega\in S_i\colon \vert\beta^{(k)}-\alpha_i\vert_\omega < 1\,, \text{ for } k=1,\ldots, N\}
\]
 and put $\hat S=\bigcup^n_{i=1} \hat S_i$; note that $\theta|_{S\setminus\hat S}=0$.  We can choose  $h(\beta^{(1)})\gg L\gg N> 21^2\log 2n $ in a way that we can assume: 

\begin{equation}\label{theepss}
\varepsilon>\frac{2}{A_1}+\frac{7\sqrt{\log 2n}}{A_1\sqrt N}-2>0
\end{equation}
By using property $(3)$ of Lemma \ref{vo8.12} and  $(\ref{theepss})$ we get:
\begin{eqnarray}\label{final_contr1}
\int_{\hat S}\theta d\mu =\int_{ S} \theta d\mu=\sum^n_{i=1}\int_{S_i}\theta|_{S_i}d\mu>A_1\sum_{i=1}^n\int_{S_i}\min_{k}\frac{-\log^-|\beta^{(k)}-\alpha_i|_\omega}{h(\beta^{(k)})}d\mu(\omega)> \nonumber \\
>A_1\left(2+\varepsilon\right)>2+\frac{7\sqrt{\log 2n}}{\sqrt N}\,.
\end{eqnarray}

Now we want to show that we can apply Theorem \ref{teoremadue}.
\begin{itemize}
\item[$\bullet$]$T(\bm{\beta})$ satisfies the h-gap condition since we can choose the $\beta^{(k)}$ such that the heights $h(\beta^{(k)})$ are separated enough by property $(2)$ of Lemma \ref{vo8.12}.
\item[$\bullet $]  Since $\vert\alpha_i-\beta^{(k)}\vert_\omega < 1$ for $\omega\in \hat S_i$,  the function $\theta$ is column bounding for $T(\bm{\beta})$  on $\hat S$. Indeed for any $k=1,\ldots, N$ and $\omega\in \hat S_i$ we have that:
$$\theta(\omega)\le \frac{-\log^- \vert \beta^{(k)}-\alpha_i\vert_{\omega}}{L+h(\beta^{(k)})}\le \frac{-\log^{-}\vert\beta^{(k)}-\alpha_i\vert_\omega}{\log\left(4^{2N!}\prod^n_{i=1}H(\alpha_i)^{\frac{2N!}{n}}\right) + h(\beta^{(k)})}$$
\end{itemize}
Therefore we can apply Theorem \ref{teoremadue} to conclude that:
$$
\int_{\hat S}\theta d \mu<2+\frac{7\sqrt{\log 2n}}{\sqrt{N}}
$$
which  contradicts inequality (\ref{final_contr1}).

\hfill{$\square$}
{}\\

Note that in the proofs of Theorems \bref{thm:main}{(A)} and \bref{thm:main}{(B)} we didn't assume Northcott property for our adelic curve.

\begin{expl}\label{Qbar}  Consider the adelic curve  $\mathbb X=(\mathbb Q,\Omega, \id)$ naturally obtained from the field $\mathbb K=\mathbb Q$ as in Example (\ref{counting}). We prove that the adelic curve $\overline{\mathbb X}=(\overline{\mathbb Q},\overline{\Omega}, \id)$ doesn't satisfy the (strong) $\mu$-equicontinuity condition by showing that the generalised Roth's theorem doesn't hold.  

The Thue equation $X^3-2Y^3=1$ has infinitely many solutions $(x_k,y_k)=\left (\sqrt[3]{2k^3+1}, k\right)$ for $k=1,2,\ldots$ in the algebraic integers. Note that by a theorem of Mahler  such solutions cannot be contained in a number field (see for instance \cite[Theorem 3.12]{Z}). Consider now the approximants $\beta_k=\frac{x_k}{y_k}=\sqrt[3]{2+\frac{1}{k^3}}$ for $\alpha=\sqrt[3]{2}$. First of all notice that
\[
h\left(\sqrt[3]{2+\frac{1}{k^3}}\right)=\log k+ O(1)
\]
whereas by using the identity $\zeta^2_3+\zeta_3+1=0$ where $\zeta_3$ is a primitive cubic root of the unity we obtain the following equality for the euclidean absolute value:

\[
-\log\left\vert \beta_k-\alpha\right\vert=-\log\left\vert \frac{x_k-\sqrt[3]{2}y_k}{k}\right\vert=-\log\left\vert\frac{x^3_k-2y^3_k}{k(x_k-\zeta_3\sqrt[3]{2}y_k)(x_k-\zeta^2_3\sqrt[3]{2}y_k)}\right\vert=
\]
 \[
 =\log k+\log\left\vert x_k-\zeta_3\sqrt[3]{2}y_k \right\vert + \log\left \vert x_k-\zeta^2_3\sqrt[3]{2}y_k \right\vert=3\log k+ O(1)
 \]
With the same strategy employed in the proof of Proposition \ref{im_cor}  it is easy to show that these calculations give a counterexample to Theorem \bref{thm:main}{(A)} for the adelic curve $\overline{\mathbb X}$. Hence we conclude that $\overline{\mathbb X}$ cannot satisfy the (strong) $\mu$-equicontinuity condition.

\end{expl}

\begin{appendices}
\section{Construction of the interpolating polynomial}\label{ap}
In this appendix we will sketch the construction of the interpolating polynomial $\delta$ of section \ref{aux}. For all the details the reader can check \cite{C}.

We employ the same notations of section \ref{aux}. We are going to construct a complicated matrix $A(\bm{X})$ depending on the following parameters:  $s, t_1,\dots, t_n\in\mathbb R$, with $0<s<1$ and $0< t_h<\frac{N}{2}$ for $h=1,\dots ,n$. The columns of the matrix are indexed by  $\bm a\in\mathcal G_N$; the rows are indexed by $\bm i_h\in\mathcal G_{t_h}$ (for any $h=1,\ldots, n$) and moreover we put $\bm i_{n+1}\in\mathcal G_s$. The order on multi-indices is the lexicographic one.

\begin{equation}\label{interpolation matrix}
A(\bm{X}):=\begin{pmatrix}
{{\bm{a}}\choose {\bm{i}}_h }\bm\alpha_h^{{\bm{a}} -{{\bm{i}}_{h}}}\\\\
{{\bm{a}}\choose {\bm{i}}_{n+1} }X^{{\bm{a}}-{\bm{i}}_{n+1}}
\end{pmatrix}_{\bm i_h,\bm a}\,\, 
\end{equation}
Note that $A(\bm X)$ has $\#(\mathcal G_s)+\sum_h \#(\mathcal G_{t_h})$ rows and $\#(\mathcal G_N)=\prod^n_{h=1}\lfloor d_h+1\rfloor$ columns, and moreover all the elements of the last $\#(\sG_{s})$ rows are monomials. Clearly we can always choose the parameters $s, t_1,\dots, t_n\in\mathbb R$ in order to obtain a matrix $A(\bm{X})$ with more rows than columns. Let's see an explicit condition that tells us when this can be achieved: the number of rows is greater than the number of columns if 
$$\#(\sG_{s})+\sum^n_{h=1}\#(\sG_{t_h})> \prod^n_{h=1}\lfloor d_h+1\rfloor
\,,$$
therefore thanks to Lemma \ref{Lemmepagina156}, for $d_1,\ldots, d_N$ very big, it is enough to have the following conditions on volumes: 
\begin{equation}\label{volume}
V(s)+\sum_{h=1}^{n}V(t_h)>1.
\end{equation}
The volumes $V(s),V(t_1),\ldots, V(t_n)$  heavily determine the algebraic properties of the matrix $A(\bm X)$, in fact we will now present a stronger condition on the quantity $V(s)+\sum_{h=1}^{n}V(t_h)$ ensuring that $A(\bm X)$ has maximal rank.
\begin{prop}\label{max_rank}
Let  $d_1,\ldots, d_N$ be big enough and let  $d_1> d_2>\ldots>d_N$. Moreover assume that 
\begin{equation*}
 \alpha^{(j)}_{h}\neq \alpha^{(j)}_{k},\,\, \forall\, j=1,\dots N,\,\,\, \forall\, h,k=1,\dots , n,\,\, h\neq k\,.
 \end{equation*}
If for $s, t_1,\dots, t_n\in\mathbb R$, with $0<s<1$ and $0< t_h<\frac{N}{2}$ 
 we have that
\begin{equation}\label{dyson_cond}
V(s)+\sum_{h=1}^{n}V(t_h)> \prod_{j=1}^{N-1} \left(1+(n-1)\sum_{i=j+1}^{N}\frac{d_i}{d_{j}}\right)
\end{equation}
then, for any  $\bm{\beta}=(\beta^{(1)},\beta^{(2)},\dots,\beta^{(N)})\in\mathbb K^N$ such that
$$
\beta^{(j)}\neq \alpha^{(j)}_{h}
\,,\,\, \forall\, j=1,\dots N,\,\,\, \forall\, h=1,\dots , n,\,,
$$
the rank of $A(\bm{\beta})$ is maximal and equal to the number of columns.
\end{prop}
\proof See \cite[Proposition 2.1]{C} and notice that it uses a version of Dyson's lemma for polynomials in many variables proved in \cite{EV}.
\endproof
One can always assume that the parameters $s,t_1,\ldots,t_n$ are chosen in a way that we always get 
\begin{equation}\label{basicinequality1}
(1+\eta)^N<V(s)+\sum_{h=1}^{n}V(t_h)<1+2N\eta
\end{equation}
It is not difficult to see (check \cite[page 159]{C}) that equation (\ref{basicinequality1}) implies (\ref{dyson_cond}). Therefore thanks to  Proposition \ref{max_rank} we can extract from $ A(\bm X)$ a square submatrix $ M(\bm X)$ of dimension $$r:=\#(\mathcal G_N)=\prod^n_{h=1}\lfloor d_h+1\rfloor\,,$$ 
which is the number of columns of $ A(\bm X)$, such that $ M(\bm X)$ has maximal rank  for any $\bm\beta$ componentwise different from any $\bm\alpha_h$. Moreover one can choose $ M(\bm X)$ in a way that contains the last $\#(\mathcal G_s)$ rows of $A(\bm X)$, since they are linearly independent for any choice of $\bm \beta\in\mathbb K^N$. The polynomial $\delta(\bm X)$ is the determinant of the matrix $M(\bm X)$. At this point all the arguments used in \cite{C} to prove the bounds about $\delta$ can be applied verbatim in our setting with the only difference that in order to estimate $h_{\mathbb X}(\delta)$ one has to take the integral over $\Omega$ instead of the summation over all places.

\end{appendices}

\bibliographystyle{alpha}
\bibliography{roth_bib.bib}

\end{document}